\documentclass[12pt,twoside]{amsart}
\usepackage{amssymb}
\usepackage{tipa}

\nonstopmode

\numberwithin{equation}{section}
\font\tengothic=eufm10 scaled\magstep 1
\font\sevengothic=eufm7 scaled\magstep 1
\newfam\gothicfam
          \textfont\gothicfam=\tengothic
          \scriptfont\gothicfam=\sevengothic

\newcommand{\fq}{\mathfrak q}
\newcommand{\fa}{\mathfrak a}
\newcommand{\fb}{\mathfrak b}

\newcommand {\PP}{\mathbb{P}}
\newcommand{\cL}{\mathcal{L}}

\newcommand{\qq}{{\mathfrak q}}

\newcommand{\fp}{{\mathfrak p}}

\def\cocoa{{\hbox{\rm C\kern-.13em o\kern-.07em C\kern-.13em o\kern-.15em A}}}

\DeclareMathOperator{\pnt}{\raise 0.5mm \hbox{\large\bf.}}

\newtheorem{theorem}{Theorem}[section]
\newtheorem{lemma}[theorem]{Lemma}
\newtheorem{proposition}[theorem]{Proposition}

\newtheorem{conjecture}[theorem]{Conjecture}

\theoremstyle{definition}
\newtheorem{remark}[theorem]{Remark}
\newtheorem{example}[theorem]{Example}

\newtheorem{notation}[theorem]{Notation}

\begin{document}

\title{On the Weak Lefschetz Property for Powers of Linear Forms}

\author[Juan Migliore]{Juan Migliore${}^*$}
\address{
Department of Mathematics, University of Notre Dame, Notre Dame, IN
46556, USA}
\email{Juan.C.Migliore.1@nd.edu}

\author[Rosa M. Mir\'o-Roig]{Rosa M. Mir\'o-Roig${}^{**}$}
\address{Facultat de Matem\`atiques, Department d'\`Algebra i Geometria, Gran Via des les Corts Catalanes 585, 08007 Barcelona, Spain}
\email{miro@ub.edu}

\author[Uwe Nagel]{Uwe Nagel${}^{***}$}
\address{Department of Mathematics,
University of Kentucky, 715 Patterson Office Tower,
Lexington, KY 40506-0027, USA}
\email{uwenagel@ms.uky.edu}

\thanks{\noindent
${}^*$ Part of the work for this paper was done while the first
author was sponsored by the National Security Agency under Grant
Number H98230-09-1-0031.\\
${}^{**}$ Partially supported by MTM2010-15256. \\
${}^{***}$ Part of the work for this paper was done while the third
author was sponsored by the National Security Agency under Grant
Number H98230-09-1-0032.\\
}


\begin{abstract} In \cite{ss}  Schenck and  Seceleanu showed that in three variables, any ideal generated by powers of linear forms has the WLP. This result contrasts with our previous examples in \cite{mmn} of ideals generated by powers of linear forms which fail the WLP. Set $R:=k[x_1,\dots,x_r]$. Assume $1<a_1 \leq  \dots  \leq a_{r+1}$.
In this paper, we concentrate our attention on almost complete intersection ideals $I =  \langle L_1^{a_1}, \dots ,L_r^{a_r},L_{r+1}^{a_{r+1}} \rangle \subset
R$ generated  by powers of    general linear forms $L_{i}$. Our approach is via the connection (thanks to Macaulay duality) to fat point ideals, together with a reduction to a smaller projective space, and we prove:

\begin{itemize}
\item If $r=4$ and $a_1=2$, then $R/I$ has the WLP

\item Assume $r=4$ and  $a_1 + a_2 + a_3 + a_4$ even.  Set $\lambda = \frac{a_1+a_2+a_3+a_4}{2} -2$. It holds
\begin{itemize}
\item If $a_5 \geq \lambda$ then  $R/I$ has the WLP.

\item If $a_5 < \lambda$ and $a_1+a_4 \geq a_2+a_3$  then $R/I$ fails the WLP.

\item  If $a_5 < \lambda$,  $a_1+a_4 < a_2+a_3$ and $2a_5 +a_1 -a_2 -a_3 -a_4 \geq 0$ then $R/I$ fails the WLP.
\end{itemize}

\item Assume $r=4$ and  $a_1 + a_2 + a_3 + a_4$ odd.  Set $\lambda = \frac{a_1+a_2+a_3+a_4-5}{2}$. It holds
\begin{itemize}
\item If $a_5 \geq \lambda -1$ then  $R/I$ has the WLP.

\item If $a_5 < \lambda -1$ and $a_1+a_4 \geq a_2+a_3$  then $R/I$ fails the WLP.

    \item If $a_5 < \lambda -1$,  $a_1+a_4 < a_2+a_3$ and $2a_5 +a_1 +3 -a_2 -a_3 -a_4 \geq 0$ then $R/I$ fails the WLP.
\end{itemize}

\item If $r=5$ and $a_1= \dots =a_6=d$, then  $R/I$ fails the WLP $ \Leftrightarrow $  $d>3$.
 \item If $r=5$, $a_1= \dots =a_5=d$ and $a_6=d+e$, $e\ge 1$.  Then:

\begin{itemize}
\item If $d$ is odd, then $R/I$ has the WLP $ \Leftrightarrow $  $e \ge
\frac{3d-5}{2}$.

\item If $d$ is even, $R/I$ has the WLP $ \Leftrightarrow $ $e \ge
\frac{3d-8}{2}$.

\end{itemize}

\item If $r=2n$, $n\ge 3$, and $a_1= \dots =a_{2n+1}=d$, then  $R/I$ fails the WLP $ \Leftrightarrow $  $d>1$.

\end{itemize}

\noindent Other examples are analyzed and we end up with a conjecture which says that if the  number of variables is odd and all powers of the same degree, say $d$, then  the WLP fails for all $d>1$. We present some further evidence for this conjecture.

\end{abstract}


\maketitle

\tableofcontents

\section{Introduction}

Ideals generated by powers of linear forms have attracted a great deal of attention recently.  For example, their Hilbert functions have been the focus of the papers \cite{AP}, \cite{sx}, \cite{hss}, among others.  In this paper we obtain further results in this direction, and relate them to the presence or failure of the Weak Lefschetz Property, which we now recall.  

Given a standard graded artinian algebra $A = R/I$, where $R = k[x_1,\dots,x_r]$ and $k$ is a field, a natural question is whether multiplication by a general linear form has maximal rank, from any degree to the next.  When this property does hold, the algebra is said to have the {\em Weak Lefschetz Property (WLP)}.  One would naively expect this property to hold, and so it is interesting to find classes of algebras where it fails, and to understand what is it about the algebra that prevents this property from holding.  There has been a long series of papers, by many authors, studying different aspects of this problem.  Even the characteristic of $k$ plays an interesting role \cite{mmn}, \cite{LZ}, \cite{BK2}.

The first result in this direction is due to R.\ Stanley \cite{stanley} and J.\ Watanabe \cite{watanabe}, who showed that, in characteristic 0,  the property holds for an artinian complete intersection generated by powers of variables.  In fact, they showed that multiplication by any power of a general linear form has maximal rank (i.e. that the {\em Strong Lefschetz Property (SLP)} holds).  Since the property is preserved after a change of variables, their result shows that it holds for any complete intersection whose generators are powers of linear forms.  By semicontinuity, it holds for a complete intersection whose generators (of arbitrary degree) are chosen generically.

There  are (at least) three natural directions suggested by this
theorem.  First, we can ask whether the property holds for
arbitrary complete intersections.  It was shown by T.\ Harima, J.\
Watanabe and the first and third authors  in \cite{HMNW} that in
two variables, {\em all} artinian algebras have the WLP.  In the
same paper, it was shown that it also holds for arbitrary artinian
complete intersections in three variables.  It remains open
whether it also holds for arbitrary complete intersections in
arbitrarily many variables.

Second, a natural question arising from the theorem of Stanley and
of Watanabe is to ask for which monomial ideals does the WLP hold
or not hold.  F.\ Zanello \cite{zanello} and H.~Brenner and A.\
Kaid \cite{BK} gave very simple examples to show that even level
monomial ideals need not have this property, and the latter gave
an example that was even an almost complete intersection (the
ideal was in a ring with three variables and had four minimal
generators).  This latter fact gave an negative answer to a
question of the first two authors \cite{mm}.  In \cite{mmn} we
gave a much more extensive study of monomial almost complete
intersections and when they fail to have the WLP.  This work was
extended by D.\ Cook II and the third author in \cite{CN}. In \cite{BMMNZ},  we showed that the only other situation where level monomial ideals have to have the WLP is 3 variables, type 2.

A third interesting problem suggested by the result of Stanley and of Watanabe is to ask when the WLP holds for powers of $\geq r+1$ linear forms, since up to a change of variables their result says that any complete intersection of powers of linear forms has the WLP.  In \cite{mmn}, we showed by example that in four variables, for $d = 3,\dots,12$, an ideal generated by the $d$-th powers of five general linear forms does not have the WLP.  On the other hand, H.\ Schenck and A.\ Seceleanu \cite{ss} then gave the surprising result that in three variables, {\em any} ideal generated by powers of linear forms has the WLP.  In contrast, Harbourne, Schenck and Seceleanu \cite{hss} have recently shown the following: Let $I = \langle \ell_1^t,\dots,\ell_n^t \rangle \subset k[x_1,\dots,x_4]$ with $\ell_i$ generic linear forms.  If $n \in \{ 5,6,7,8\}$ then the WLP fails, respectively, for $t \geq \{ 3,27,140,704\}$.

In this paper, we study the  WLP for quotients $k[x_1,\cdots ,x_r]/I$ where $I$ is an almost complete intersection ideal
 generated  by powers of general linear forms. As a main tool we first use the inverse system dictionary to relate an ideal $I\subset k[x_1, \cdots ,x_r]$ generated by powers of linear forms to an ideal of fat points in $\PP^{r-1}$, and  then we show that the WLP problem of an ideal generated by powers of linear forms is  closely connected to the geometry  of the linear system  of hypersurfaces in $\PP^{r-2}$ of fixed degree  with preassigned multiple points.

Let us briefly explain how this paper is organized. We begin in  Section \ref{gen appr}
explaining the tools that  are applied
throughout the paper.  First, we recall  a result of Emsalem and Iarrobino which gives  a duality between powers of linear forms and ideals of fat points in $\mathbb P^{r-1}$. Then, we reduce our WLP problem to one of computing the Hilbert function of $n$ general fat points in $\mathbb P^{r-2}$
or, equivalently, to compute the dimension of the linear system of hypersurfaces in $\PP^{r-2}$ of degree $d$  having some points of fixed multiplicity. Moreover, using Cremona transformations, one can relate two different linear systems (see  \cite{LU}, or \cite{Dumnicky}, Theorem 3).

In Section \ref{4-variables},  we consider the case of 4 variables
and  we give a fairly complete  answer about the failure of the
WLP  for $I= \langle
L_1^{a_1},L_2^{a_2},L_3^{a_3},L_4^{a_4},L_5^{a_5} \rangle \subset
k[x_1,x_2,x_3,x_4]$, where $L_{i}$ are  general linear forms and
$2\leq a_1 \leq a_2 \leq a_3 \leq a_4 \leq a_5$. In Section
\ref{5-variables}, we deal with 5 variables and we completely
determine  when an ideal  generated by uniform (resp. almost
uniform) powers of six linear forms fails WLP. We add some
examples to illustrate  that our methods extend beyond the
mentioned results. The main result of Section \ref{uniform-powers}
is Theorem \ref{2n var uniform case} where we give a complete
answer to the uniform case when the number of variables is even;
in particular, we solve  Conjecture 5.5.2 in \cite{hss} when the
number of variables is even. The case of an odd number of
variables is left as an open conjecture and we present some
evidence for this conjecture, including the case of seven
variables.

Finally it is worthwhile to point out that the approach of this work can be applied to many other situations, in particular, when
the generators do not all have the same degree, but that the calculations quickly become
overwhelming.


\section{General approach} \label{gen appr}

Let $R = k[x_1,\dots,x_r]$ be a polynomial ring, where $k$ is a field of characteristic zero.

\begin{notation}
Throughout this paper, when $m$ is any integer, we will denote
\[
[m]_+ = \max \{ m,0 \}.
\]
Also, for a standard graded algebra $A$ we denote by $h_A$ the Hilbert function
\[
h_A(t) = \dim_k [A]_t .
\]
\end{notation}

For any artinian ideal $I \subset R$ and a general linear form $\ell \in R$, the exact sequence
\[
\cdots \rightarrow [R/I]_{m-1} \stackrel{\times \ell}{\longrightarrow} [R/I]_m \rightarrow [R/(I,\ell)]_m \rightarrow 0
\]
gives, in particular, that the multiplication by $\ell$  will fail to have maximal rank exactly when
\begin{equation}
  \label{eq:max-rank-I}
\dim_k [R/(I,\ell)]_m \neq \max\{ \dim_k [R/I]_m - \dim_k [R/I]_{m-1} , 0 \};
\end{equation}
in that case, we will say that  $R/I$  fails  the WLP in degree $m$. In several of the papers mentioned above, this failure was studied via an examination of the splitting type of the first syzygy bundle of $I$.  For powers of linear forms, we give an alternative approach, which we will implement in the subsequent sections.

We first recall a result of Emsalem and Iarrobino giving a duality between powers of linear forms and ideals of fat points in $\mathbb P^{n-1}$.  We only quote Theorem I in \cite{EI} in the form that we need.

\begin{theorem}[\cite{EI}]   \label{thm:inverse-system}
Let $\langle L_1^{a_1} ,\dots,L_n^{a_n} \rangle \subset R$ be an ideal generated by powers of $n$ general linear forms.  Let $\wp_1, \dots, \wp_n$ be the ideals of $n$ general points in $\mathbb P^{r-1}$.  (Each point is actually obtained explicitly from the corresponding linear form by duality.)  Choose positive integers $a_1,\dots,a_n$.  Then for any integer $j \geq \max\{ a_i + 1 \}$,
\[
\dim_k \left [R/ \langle L_1^{a_1}, \dots, L_n^{a_n} \rangle  \right ]_j =
\dim_k \left [ \wp_1^{j-a_1 +1} \cap \dots \cap \wp_n^{j-a_n+1} \right ]_j .
\]
\end{theorem}

Now, we observe that the ideal $(I,\ell)$ is also an ideal generated by powers of linear forms!  We conclude that if $\wp$ is the ideal of the point dual to $\ell$ then
\[
\dim_k \left [R/ \langle L_1^{a_1}, \dots, L_n^{a_n}, \ell  \rangle \right ]_j =
\dim_k \left [  \wp_1^{j-a_1 +1} \cap \dots \cap \wp_n^{j-a_n+1} \cap \wp^j  \right ]_j .
\]
Consider the points $P_1,\dots, P_n, P$  in $\mathbb P^{r-1}$ defined by the ideals $\wp_1,\dots,\wp_n,\wp$ respectively.  Let $\lambda_i$ be the line joining $P$ to $P_i$, and let $H = \mathbb P^{r-2}$ be a general hyperplane defined by a linear form $L_H$.  Let $Q_i$ be the point of intersection of $\lambda_i$ with $H$.  For any positive integer $m$, we will denote by $\lambda_i^m$ the curve with defining ideal $I_{\lambda_i}^m$, and notice that $\lambda_i^m$ is arithmetically Cohen-Macaulay.  Thus the hyperplane section  of $\lambda_i^m$ by $H$ has saturated ideal $\qq_i^m = (I_{\lambda_i}^m,L_H)/(L_H)$ in the coordinate ring $R/(L_H)$ of $H$, and $\qq_i^m$ defines a fat point  $Q_i^m$ in $H = \mathbb P^{r-2}$.  The curve $Y = \lambda_1^{j-a_1 +1} \cup \dots \cup \lambda_n^{j-a_n+1}$ is the cone over $Q_1^{a_1} \cup \dots \cup Q_n^{a_n}$, and thus is also arithmetically Cohen-Macaulay.

\begin{proposition} \label{reduce proj sp}
\[
\begin{array}{rcl}
\dim_k \left [ R/ \langle L_1^{a_1}, \dots, L_n^{a_n}, \ell  \rangle  \right ]_j & = &
\dim_k \left [ \wp_1^{j-a_1 +1} \cap \dots \cap \wp_n^{j-a_n+1} \cap \wp^j \right ]_j  \ \ \ \hbox{\rm (in $k[x_1,\dots,x_{r}]$)} \\
& = & \dim_k \left [ \qq_1^{j-a_1 +1} \cap \dots \cap \qq_n^{j-a_n+1} \right ]_j \ \ \ \hbox{\rm (in $k[x_1,\dots,x_{r-1}]$)} .
\end{array}
\]

\end{proposition}

\begin{proof}
The first equality is the result of Emsalem and Iarrobino mentioned above.
Without loss of generality let $P = [0,\dots,0,1]$, with defining ideal $\wp = \langle x_1,\dots,x_{r-1}\rangle$ and assume that $H$ is defined by $x_r = 0$.  Any form $F \in \left [ \wp_1^{j-a_1 +1} \cap \dots \cap \wp_n^{j-a_n+1} \cap \wp^j \right ]_j $ involves only the variables $x_1,\dots, x_{r-1}$, being in $[\wp^j]_j$.  Thus $F \in [I_Y]_j$, so viewing $F$ in $k[x_1,\dots,x_{r-1}]$ we see that $F$ vanishes on the hyperplane section of $Y$, i.e. $F \in [ \qq_1^{j-a_1 +1} \cap \dots \cap \qq_n^{j-a_n+1} ]_j $.  On the other hand, for any $F \in [ \qq_1^{j-a_1 +1} \cap \dots \cap \qq_n^{j-a_n+1} ]_j $, viewing $F$ in $k[x_1,\dots,x_r]$ implies that $F$ vanishes on $Y$, so $F$ vanishes on the subscheme of $Y$ defined by $\wp_1^{j-a_1 +1} \cap \dots \cap \wp_n^{j-a_n+1}$, and $F$ is in $\wp^j$ since it involves only $x_1,\dots,x_{r-1}$ and has degree $j$.
\end{proof}

In this way, we reduce our WLP problem to one of computing the Hilbert function of $n$ general fat points in $\mathbb P^{r-2}$. From now on, we will denote by
$${\mathcal L}_{r-2}(j; j-a_1+1,j-a_2+1,\cdots ,j-a_n+1)$$
the linear system $[ \qq_1^{j-a_1 +1} \cap \dots \cap \qq_n^{j-a_n+1} ]_j\subset [k[x_1,\cdots ,x_{r-1}]]_j $.
In order to simplify notation, we use superscripts to indicate repeated entries. For example,
$\cL_3(j; 5^2, 2^3) = \cL_3 (j; 5, 5, 2, 2, 2)$.

Notice that, for every linear system $\cL_r (j; a_1,\ldots,a_n)$, one has
\[
\dim_k  \cL_r (j; a_1,\ldots,a_n) \ge \max \left \{0, \binom{j+r}{r} - \sum_{i=1}^n \binom{a_i + r-1}{r} \right \},
\]
where the right-hand side is called the {\em expected dimension} of the linear system. If the inequality is strict, then the linear system $\cL_r (j; a_1,\ldots,a_n)$ is called {\em special}. It is a difficult problem to classify the special linear systems.

Using Cremona transformations, one can relate two different linear systems (see \cite{Nagata}, \cite{LU}, or \cite{Dumnicky}, Theorem 3).

\begin{lemma}
  \label{lem:Cremona}
Let $n > r \ge 2$ and  $j, a_1,\ldots,a_n$ be non-negative integers, and set $m = (r-1) j - (a_1 + \cdots + a_{r+1})$. If $a_i + m \ge 0$ for all $i = 1,\ldots,r+1$, then
\[
\dim_k \cL_r (j; a_1,\ldots,a_n) = \dim_k \cL_r (j + m; a_1 +m,\ldots,a_{r+1} +m , a_{r+2},\ldots,a_n).
\]
\end{lemma}

Following \cite{DL}, the linear system $\cL_r (j; a_1,\ldots,a_n)$ is said to be in {\em standard form} if \[
(r-1) j \ge a_1 + \cdots + a_{r+1} \quad \text{and} \quad a_1 \ge \cdots \ge a_n \ge 0.
\]
If $r=2$, then every linear system in standard form is non-special. This is no longer true if $r \ge 3$. However, De Volder and Laface \cite{DL} were able to compute the speciality in the case of at most 8 fat points in $\PP^3$. We state their result only in the form we need it.

\begin{theorem}[\cite{DL}, Theorem 5.3]
  \label{thm:fat-points-P3}
If the linear system $\cL_3 (j; a_1,\ldots,a_6)$ is in standard form, then
\[
\dim_k \cL_3 (j; a_1,\ldots,a_6) = \max \left \{0, \binom{j+r}{r} - \sum_{i=1}^6 \binom{a_i + r-1}{r} \right \} + \sum_{i=2}^6 \binom{t_i + 1}{3},
\]
where $t_i = a_1 + a_i - j$.
\end{theorem}

Notice that we always use the vector space dimension of the linear system rather than the dimension of its projectivization and that we adjusted the formula for the expected dimension. Furthermore, we always use the convention that a binomial coefficient $\binom{a}{r}$ is zero if $a < r$.
\smallskip

In this note, we are interested in certain almost complete
intersections.  Then one can compute the right-hand side of
Inequality  \eqref{eq:max-rank-I}.

\begin{lemma}
  \label{lem:expected-Hilb}
Let $I = \langle L_1^{a_1} ,\dots,L_{r+1}^{a_{r+1}} \rangle \subset R$ be an almost complete intersection generated by powers of $r+1$ general linear forms. Then, for each integer $j$,
\[
\dim_k [R/I]_j - \dim_k [R/I]_{j-1} = \big [ h_A (j) - h_A (j- a_{r+1}) \big ]_+ - \big [ h_A (j- 1) - h_A (j - 1- a_{r+1}) \big ]_+ ,
\]
where $A = R/\langle L_1^{a_1} ,\dots,L_r^{a_{r}} \rangle$. Furthermore, if $j \le \frac{1}{2} a_{r+1} + \frac{1}{2} \sum_{i=1}^r (a_i - 1)$,  then the formula simplifies to
\[
\dim_k [R/I]_j - \dim_k [R/I]_{j-1} = \big [ h_A (j) - h_A (j- 1)  \big ] - \big [h_A (j- a_{r+1})  - h_A (j - 1- a_{r+1}) \big ].
\]
\end{lemma}

\begin{proof}
Considering multiplication by $l_{r+1}^{a_{r+1}}$ on $A$, the first equation follows because the complete intersection $A$ has the SLP according to \cite{stanley} or \cite{watanabe}. The latter also implies that the Hilbert function of $A$ is unimodal. Its midpoint is $\frac{1}{2} \sum_{i=1}^r (a_i - 1)$. Thus, the differences in brackets in the first formula are not negative if $j \le \frac{1}{2} a_{r+1} + \frac{1}{2} \sum_{i=1}^r (a_i-1)$, proving the second formula.
\end{proof}

Notice that the Hilbert function of the complete intersection $A$ can be computed using the Koszul complex that provides its minimal free resolution.


\section{Powers of linear forms in four variables} \label{4-variables}

In this section we let $R = k[x_1,x_2,x_3,x_4]$, where $k$ is a field of characteristic zero.  Our main result will be to determine, in almost all cases, when an ideal generated by powers of five general linear forms has the WLP.  To this end, without loss of generality we set $I = \langle x_1^{a_1},x_2^{a_2},x_3^{a_3},x_4^{a_4},L^{a_5} \rangle$, where $L$ is a general linear form and $a_1 \leq a_2 \leq a_3 \leq a_4 \leq a_5$.

\begin{lemma} \label{indep cond}
Let $P_1,\dots, P_6$,  be points in $\mathbb P^2$ in linear general position.  Assign multiplicities $m_1, \dots, m_6 $ respectively to the points, with $0 \leq m_1 \leq \dots \leq m_6$.  (In particular, taking some of the $m_i = 0$ allows us to consider fewer than six points.) Assume that $d \geq m_{5}+m_6$ and that $2d \geq \sum_{i=2}^6 m_i$.  Then the fat point scheme $Z = m_1 P_1 + \cdots + m_6 P_6$ imposes independent conditions on curves of degree $d$.
\end{lemma}

\begin{proof}
Let $X$ be the rational surface obtained by blowing up  $\mathbb P^2$ at the points $P_i$.  Let $L, E_1,E_2,E_3,E_4,E_5,E_6$ be the standard basis of the divisor class group of $X$, i.e. $L$ is the pullback of the class of a line in $\mathbb P^2$ and $E_1,\dots,E_6$ are the exceptional divisors.  Under the stated assumptions $d \geq m_5 +m_6$ and $2d \geq \sum_{i=2}^6 m_i$, the divisor $dL - m_1E_1 - m_2 E_2 - m_3 E_3 - m_4 E_4-m_5E_5-m_6E_6$ is numerically effective (nef) (cf.\ Theorem 3.4 of \cite{dirocco}).  Then Theorem 2.3 and Remark 2.4 of \cite{GHM} show that the Castelnuovo-Mumford regularity of $I_Z$ is $\leq d+1$.  This implies the claimed result.
\end{proof}

\begin{theorem} \label{four var even}
Let $L$ be a general linear form and let $I = \langle x_1^{a_1},x_2^{a_2},x_3^{a_3},x_4^{a_4},L^{a_5} \rangle$.
Assume that $a_1 + a_2 + a_3 + a_4$ is even.  Let $\lambda = \frac{a_1+a_2+a_3+a_4}{2} -2$.
\begin{itemize}
\item[(i)] If $a_5 \geq \lambda$ then the ring $R/I$ has the WLP.

\item[(ii)] If $a_5 < \lambda$ and $a_1+a_4 \geq a_2+a_3$ then $R/I$ fails the WLP from degree $\lambda-1$ to degree~$\lambda$.

\item[(iii)] If $a_5 < \lambda$, $a_1+a_4 < a_2+a_3$ and $2a_5 +a_1 -a_2 -a_3 -a_4 \geq 0$ then $R/I$ fails the WLP from degree $\lambda-1$ to degree $\lambda$.
\end{itemize}
The hypotheses of (ii) and of (iii) both force $a_1 \geq 3$.
\end{theorem}

\begin{proof}
Let $\ell$ be a general linear form.  Letting $J = \langle x_1^{a_1},x_2^{a_2},x_3^{a_3},x_4^{a_4} \rangle$, consider the commutative diagram
 \[
 \begin{array}{cccccccccc}
 [R/J]_t & \stackrel{\times \ell}{\longrightarrow} & [R/J]_{t +1} \\
 \downarrow && \downarrow \\

 [R/I]_t & \stackrel{\times \ell}{\longrightarrow} & [R/I]_{t +1}
  \end{array}
  \]
  (where the vertical arrows are the natural restriction).  If we know that the multiplication on the top row is surjective, then we immediately conclude surjectivity on the bottom row.   Notice that $2\lambda$ is the socle degree of the artinian complete intersection \linebreak $R/(x_1^{a_1},x_2^{a_2},x_3^{a_3},x_4^{a_4})$, so $\lambda$ is the midpoint.

 First assume that $a_5 > \lambda$.  Then clearly $R/I$ and $R/J$ coincide in degrees $\leq \lambda$, so we have injectivity (by the Stanley-Watanabe result) for $(\times \ell) : [R/I]_t \rightarrow [R/I]_{t+1}$ for all $t \leq \lambda-1$.  When $t \geq \lambda$ we have surjectivity for $R/J$, so by the above result we also have it for $R/I$.

Now assume that $a_5 = \lambda$.  We wish to show that $R/I$ has the WLP.  Again surjectivity is immediate for $t \geq \lambda$, and injectivity is immediate for $t \leq \lambda-2$.  When  $t = \lambda-1$ we consider two cases.  If $\dim [R/J]_{\lambda-1} = \dim [R/J]_\lambda$ then by Stanley-Watanabe we have surjectivity for $ (\times \ell) : [R/J]_{\lambda-1} \rightarrow [R/J]_\lambda$, so the same holds for $R/I$.  If $\dim [R/J]_{\lambda-1} < \dim [R/J]_\lambda$, the image of $[R/J]_{\lambda-1}$ in $[R/J]_\lambda$ under multiplication by $\ell$ is not surjective.  Hence the vector space $[\langle x_1^{a_1},x_2^{a_2},x_3^{a_3},x_4^{a_4}, \ell \rangle]_\lambda$ is not all of $R_\lambda$.  But in characteristic zero the $\lambda$-th powers of linear forms span $[R]_\lambda$.  Thus for a general $L$, the image of $L^\lambda$ in $[R/J]_\lambda$ is outside the image of $[R/J]_{\lambda-1}$ in $[R/J]_\lambda$.  Thus $(\times \ell) : [R/I]_{\lambda-1} \rightarrow [R/I]_\lambda$ is injective, since it is for $R/J$ in that degree.  This completes the proof of (i).
\smallskip

We now prove (ii).  Our assumptions now are that $a_1+a_2+a_3+a_4$ is even, $a_5 < \lambda$, and $a_1+a_4 \geq a_2+a_3$.  We first note that the hypotheses force $a_1 \geq 3$, since if $a_1 = 2$ then  $a_5 < \lambda = \frac{a_2+a_3+a_4-2}{2} \leq \frac{2 + 2a_4 -2}{2} = a_4$.

We will show that the multiplication by $\ell$ fails to have maximal rank from degree $\lambda-1$ to degree $\lambda$, although sometimes it is a failure of injectivity and sometimes it is a failure of surjectivity.   (See Example \ref{sometimes inj, surj}.)

As mentioned in Section \ref{gen appr}, the task is to show that
\begin{equation}\label{to show}
\dim [R/(I,\ell)]_\lambda >  [\dim[R/I]_\lambda - \dim[R/I]_{\lambda-1}]_+.
\end{equation}
We will compute the left-hand side using the approach of Section \ref{gen appr} and the right-hand side using Lemma \ref{lem:expected-Hilb}.

By Proposition \ref{reduce proj sp} and Theorem \ref{thm:inverse-system}, we have
\begin{eqnarray}
   \label{desired dim}
\lefteqn{
  \dim_k [R/(I,\ell)]_{\lambda} } \nonumber \\
& = &
\dim_k [\qq_1^{\lambda-a_1+1} \cap \qq_2^{\lambda-a_2+1} \cap \qq_3^{\lambda-a_3+1} \cap \qq_4^{\lambda-a_4+1} \cap \qq_5^{\lambda-a_5+1} ]_\lambda \\[1ex]
& = & \dim_k {\mathcal L}_2(\lambda; \lambda-a_1+1,\lambda-a_2+1,\lambda-a_3+1,\lambda-a_4+1,\lambda-a_5+1). \nonumber
\end{eqnarray}
Notice that
\[
\lambda - a_1 + 1 \geq \dots \geq \lambda-a_5+1 \geq 2.
\]
The vector space defines a linear system, and we want to find its dimension.  The first step is to understand the one-dimensional base locus, which has components of degree 2 and of degree 1.  We will use Bezout's theorem to formally reduce the degree of the polynomials and the order of vanishing at the points, without changing the dimension of the linear system.  If the result has dimension zero then (\ref{desired dim}) is also zero.

Let $F_\lambda \in {\mathcal L}_2(\lambda; \lambda-a_1+1,\lambda-a_2+1,\lambda-a_3+1,\lambda-a_4+1,\lambda-a_5+1) $.  There is a unique quadratic polynomial $F_2$ vanishing at the five general points.  By Bezout's theorem, if $F_2$ is not a factor of $F_\lambda$ then it intersects $F_\lambda$ with multiplicity $2\lambda$.  On the other hand, considering the multiplicities at the five points, we get that the curves meet with multiplicity at least $5\lambda- \sum_{i=1}^5 a_i + 5$.  But
\[
5\lambda- \sum_{i=1}^5 a_i + 5 \leq 2\lambda \ \ \Longleftrightarrow \ \  a_5 \geq \frac{a_1+a_2+a_3+a_4}{2} -1 = \lambda+1,
\]
 a contradiction.  Hence $F_2$ is a factor of $F_\lambda$.  We next want to know what power of $F_2$ divides $F_\lambda$.  Thus we want to know the smallest $i$ for which
\[
2(\lambda-2i) \geq (\lambda-a_1+1-i) + \cdots + (\lambda-a_5+1-i),
\]
and this is equivalent to $i \geq \lambda-a_5+1 $.  Thus we have
\[
F_\lambda = F_2^{\lambda-a_5+1} \cdot F_{2a_5-\lambda-2}
\]
where $F_{2a_5 - \lambda -2} \in {\mathcal L}_2(2a_5-\lambda-2; a_5-a_1,a_5-a_2,a_5-a_3,a_5-a_4)$.  (Notice that now there are only four points, and some of these multiplicities might even be zero.)

Now we consider linear factors coming from the lines joining two of these four points.  There are six such lines; we denote by $L_{ij}$, $1\le i <j\le 4$, the line (as well as the linear form) passing through the points with multiplicities $a_5 - a_i$ and $a_5-a_j$.  Notice that if $a_5 - a_i > 2a_5 - \lambda -2$ then there are no such forms $F_{2a_5 - \lambda-2}$, so the desired dimension is zero.

Arguing in a similar manner as above, we obtain that if $(a_5 - a_i) + (a_5 - a_j) > 2a_5 - \lambda -2$ then $L_{ij}$ appears as a factor
\[
(a_5 - a_i) + (a_5-a_j) - (2a_5 - \lambda -2) = \lambda +2-a_i-a_j
\]
times.  Thus, letting
\[
A_{ij} = [\lambda+2-a_i-a_j]_+ =  \max \{\lambda+2-a_i-a_j ,0 \},
\]
we see that (formally)
\begin{equation} \label{both cases}
F_\lambda = F_2^{\lambda-a_5+1} \cdot \prod_{1 \leq i<j\leq 4} L_{ij}^{A_{ij}} \cdot G, \ \ \ \hbox{ where } \deg G = 2a_5 - \lambda -2 - \sum_{1 \leq i < j \leq 4} A_{ij}.
\end{equation}
Notice that
\begin{equation}\label{constraint}
\hbox{if } 2a_5 - \lambda - 2 < \sum_{1 \leq i < j \leq 4} A_{ij},  \hbox{ then }  \dim[\langle x_1^{a_1},x_2^{a_2},x_3^{a_3},x_4^{a_4},L^{a_5},\ell \rangle ]_{\lambda} = 0.
\end{equation}

Now, since we have assumed $a_1 + a_4 \geq a_2 + a_3$, observe that
\begin{equation} \label{pos neg}
\begin{array}{rclcrcl}
a_1+a_4-a_2-a_3 & \geq & 0 \hbox{\hspace{1.5cm}} & a_2 +a_3 - a_1 - a_4 & \leq & 0 \\
a_2+a_4-a_1-a_3 & \geq & 0 \hbox{\hspace{1.5cm}} & a_1 +a_3 - a_2 - a_4 & \leq & 0 \\
a_3+a_4-a_1-a_2 & \geq & 0 \hbox{\hspace{1.5cm}} & a_1 +a_2 - a_3
- a_4 & \leq & 0 .
\end{array}
\end{equation}
We claim that after removing the one-dimensional base locus, we obtain a set of $\leq 4$ fat points of uniform multiplicities (possibly 0).  We have already seen that after removing the powers of $F_2$ we are left with the problem of finding the dimension of the forms of degree $2a_5 - \lambda - 2$ passing through four general fat points with multiplicity $a_5 - a_1, \dots, a_5-a_4$.  To compute this, we have to determine precisely what is left when we remove the lines.  At each of the four points we compute the multiplicity of the fat point that remains after we remove the (multiple) lines passing through it in the base locus (which do not contribute to the desired dimension).
\begin{eqnarray*}
\lefteqn{a_5 - a_4 - \sum_{i=1,2,3} [\lambda+2-a_i-a_4]_+  } \\[1ex]
&  = &  a_5 - a_4 \hbox{$ - \left [ \frac{-a_1 + a_2 + a_3 - a_4}{2} \right ]_+ - \left [ \frac{a_1 - a_2 + a_3 - a_4}{2} \right ]_+ - \left [ \frac{a_1 + a_2 - a_3 - a_4}{2} \right ]_+ $}\\[1ex]
&  = & a_5 - a_4.
\end{eqnarray*}
Similarly,
\[
\begin{array}{rcl}
a_5 - a_3 - \sum_{i=1,2,4} [\lambda+2-a_i-a_3]_+  & = & a_5 - a_4 \\ \\
a_5 - a_2 - \sum_{i=1,3,4} [\lambda+2-a_i-a_2]_+  & = & a_5 - a_4 \\ \\
a_5 - a_1 - \sum_{i=2,3,4} [\lambda+2-a_i-a_1]_+  & = & a_5 - a_4 .\\ \\
\end{array}
\]
Concluding, we want to find the dimension of the vector space of forms of degree
\[
2a_5 - \lambda - 2 - \sum_{1 \leq i < j \leq 4} A_{ij} = 2(a_5-a_4)
\]
 passing through four points with multiplicity $a_5-a_4$.  By Lemma \ref{indep cond}, this dimension is $(a_5-a_4+1)$ (which in particular is at least 1).  That is,
 \begin{equation}\label{value case 1}
 \dim[\langle x_1^{a_1},x_2^{a_2},x_3^{a_3},x_4^{a_4},L^{a_5},\ell \rangle ]_{\lambda} = a_5-a_4+1 \geq 1.
\end{equation}

Using Lemma \ref{lem:expected-Hilb}, we now compute the ``expected'' dimension, i.e. the right-hand side of (\ref{to show}).  Let $A = R/J$, where $J = \langle x_1^{a_1},x_2^{a_2},x_3^{a_3},x_4^{a_4} \rangle$.   In the following, recall also that $\lambda$ is the midpoint of the $h$-vector of $R/J$, so $0 \leq h_{R/I}(\lambda) = h_A(\lambda) - h_A(\lambda-a_5) $.  Observe that $a_4 < a_1 + a_2 + a_3$ since otherwise
\[
a_5 - \lambda \geq a_4 - \lambda = \frac{a_4 - a_1 - a_2 - a_3}{2} + 2 > 0,
\]
contradicting our assumption that $a_5 < \lambda$.  This implies, in particular, that for determining $h_A (\lambda)$ by using the Koszul resolution of $A$, we only need to consider  the free modules up to $\bigoplus_{1\le i<l\le 4}R(-a_i-a_j)$.

From (\ref{pos neg}) we see that only $\lambda - a_2 - a_3 +2$, $\lambda - a_1 - a_3+2$ and $\lambda - a_1 - a_2+2$ can be positive.  Similarly, we have $\lambda - a_5 - a_i +2 \leq 0$ since $a_1 + a_4 \geq a_2 + a_3$.
Then the ``expected dimension" is
\begin{eqnarray*}
\lefteqn{
\big [\dim [R/I]_\lambda - \dim [R/I]_{\lambda-1]} \big ]_+ } \\[1ex]
& = & \big [h_A(\lambda) - h_A (\lambda-a_5) - h_A(\lambda-1) + h_A(\lambda-a_5-1) \big ]_+ \\[1ex]
& = & \left [ \binom{\lambda+2}{2} - \sum_{i=1}^4 \binom{\lambda-a_i +2}{2} + \sum_{1 \leq i < j \leq 3} \binom{\lambda-a_i - a_j +2}{2}  - \binom{\lambda-a_5+2}{2}   \right ]_+.
\end{eqnarray*}
If this is zero then clearly the actual dimension exceeds the expected one and we are done.  If not, one verifies (for example with \cocoa\ \cite{cocoa}) that
\begin{eqnarray*}
\lefteqn{
\dim_k [R/(I,\ell)]_{\lambda}  - \big [\dim [R/I]_\lambda - \dim [R/I]_{\lambda-1]} \big ]_+ } \\
& = & (a_5-a_4+1)  \\
&& - \left (  \binom{\lambda+2}{2} - \sum_{i=1}^4 \binom{\lambda-a_i +2}{2} + \sum_{1 \leq i < j \leq 3} \binom{\lambda-a_i - a_j +2}{2}  - \binom{\lambda-a_5+2}{2}  \right ) \\
& = & \binom{\lambda+1-a_5}{2},
\end{eqnarray*}
and this last binomial coefficient is at least 1.  Thus, in either case we have shown Inequality \eqref{to show}, which completes the proof of (ii).
\smallskip

Finally we prove (iii).  Our assumptions now are that $a_1+a_2+a_3+a_4$ is even, $a_5 < \lambda$, $a_1+a_4 < a_2+a_3$ and $2a_5 + a_1 - a_2 - a_3 -a_4 \geq 0$.  The calculations from (ii) continue to be valid up to (\ref{both cases}) and (\ref{constraint}).  We first note that the hypotheses again force $a_1 \geq 3$ since if $a_1 = 2$ then $a_5 < \lambda = \frac{a_2+a_3+a_4-2}{2} \leq  \frac{2 + 2a_5 - 2}{2} = a_5$.

In our current situation, though, observe that
\begin{equation} \label{plus minus (iii)}
\begin{array}{rclcrcl}
a_3+a_4-a_1-a_2 & > & 0 \hbox{\hspace{1.5cm}} & a_1 +a_2 - a_3 - a_4 & < & 0 \\
a_2+a_4-a_1-a_3 & > & 0 \hbox{\hspace{1.5cm}} & a_1 +a_3 - a_2 - a_4 & < & 0 \\
a_2+a_3-a_1-a_4 & > & 0 \hbox{\hspace{1.5cm}} & a_1 +a_4 - a_2
- a_3 & < & 0 .
\end{array}
\end{equation}
Now we examine the linear system that remains when we remove the one-dimensional base locus.
As in (ii), we obtain a linear system of curves of degree $2a_5 - \lambda - 2 - \sum_{1 \leq i < j \leq 4} A_{ij}$.  Now we compute the order of vanishing at the four points:
\begin{eqnarray*}
\lefteqn{
a_5 - a_4 - \sum_{i=1,2,3} [\lambda+2-a_i-a_4]_+ } \\
& = & a_5 - a_4 - \hbox{$ \left [ \frac{-a_1 + a_2 + a_3 - a_4}{2} \right ]_+ - \left [ \frac{a_1 - a_2 + a_3 - a_4}{2} \right ]_+ - \left [\frac{a_1 + a_2 - a_3 - a_4}{2} \right ]_+ $}\\
& = & a_5 - \frac{-a_1+a_2+a_3+a_4}{2}.
\end{eqnarray*}
By the additional hypothesis $2a_5 + a_1 - a_2 - a_3 -a_4 \geq 0$, this order of vanishing is $\geq 0$.  A similar computation gives that the order of vanishing at the other three points is the same.  Thus Lemma \ref{indep cond} shows that
\[
 \dim[\langle x_1^{a_1},x_2^{a_2},x_3^{a_3},x_4^{a_4},L^{a_5},\ell \rangle ]_{\lambda}   = a_5 - \frac{-a_1+a_2+a_3+a_4}{2} +1 \geq 1.
\]
The computation of the ``expected" dimension is very similar to what we did above.
From (\ref{plus minus (iii)}) we see that only $\lambda - a_1 - a_2 +2$, $\lambda - a_1 - a_3+2$ and $\lambda - a_1 - a_4+2$ can be positive.  We again have $a_4 < a_1 + a_2 + a_3$.  Thus, Lemma \ref{lem:expected-Hilb} provides
\begin{eqnarray*}
\lefteqn{
\big [\dim [R/I]_\lambda - \dim [R/I]_{\lambda-1]} \big ]_+ } \\[1ex]
& = & \big [h_A(\lambda) - h_A (\lambda-a_5) - h_A(\lambda-1) + h_A(\lambda-a_5-1) \big ]_+ \\[1ex]
& = &  \left [ \binom{\lambda+2}{2} - \sum_{i=1}^4 \binom{\lambda-a_i +2}{2} +
 \sum_{i=2}^4 \binom{\lambda-a_1 - a_i +2}{2}  - \binom{\lambda-a_5+2}{2}   \right
 ]_+.
\end{eqnarray*}
As above, if this is zero then we are done.  If not, one  verifies (for example with \cocoa\ ) that
\begin{eqnarray*}
\lefteqn{
\dim_k [R/(I,\ell)]_{\lambda}  - \big [\dim [R/I]_\lambda - \dim [R/I]_{\lambda-1]} \big ]_+ } \\
& = & a_5 - \frac{-a_1+a_2+a_3+a_4}{2} +1  \\
&& - \left (  \binom{\lambda+2}{2} - \sum_{i=1}^4 \binom{\lambda-a_i +2}{2} + \sum_{i=2}^4 \binom{\lambda-a_1 - a_i +2}{2} - \binom{\lambda-a_5+2}{2}  \right ) \\
& = & \binom{\lambda+1-a_5}{2},
\end{eqnarray*}
and this last binomial coefficient is at least 1, establishing Inequality \eqref{to show}.    This completes the proof of (iii).
\end{proof}

\begin{example}\label{sometimes inj, surj}
To illustrate that sometimes it is injectivity that fails and sometimes it is surjectivity, consider the following (produced using \cocoa).

When $a_1 = 5, a_2 = 7, a_3 = 8, a_4 = 10, a_5 = 10$ we get $\dim [R/I]_{12} = 225$, $\dim[R/I]_{13} = 220$, but the image under multiplication by a general linear form has dimension 219.

When $a_1 = 5, a_2 = 7, a_3 = 8, a_4 = 10, a_5 = 12$ we get $\dim [R/I]_{12} = 234$, $\dim[R/I]_{13} = 236$, but the image under multiplication by a general linear form has dimension 233.
\end{example}

Let us now discuss cases that are not covered by  Theorem \ref{four
var even}, still assuming that $a_1 + a_2 + a_3 + a_4$ is even.

\begin{remark}
As above, let  $I = \langle
x_1^{a_1},x_2^{a_2},x_3^{a_3},x_4^{a_4},L^{a_5} \rangle$. Assume
$2a_5 +a_1 -a_2 -a_3 -a_4 < 0$ and $a_1 \ge 3$. Then:
\begin{itemize}
\item[(i)] If $a_1 = 2$ then $R/I$ has the WLP (see Theorem
\ref{thm:wlp-four-variables}).
\item[(ii)] If $a_1 = 3$, then in the following cases,
 $R/I$ fails to have the WLP, and the  failure is in degree
 $\lambda-1$ and it fails by 1.
\begin{itemize}
\item $(3,9,m,m,m)$ for $m \geq 9 $
\item $(3,10,m,m+1,m+1)$ for $m \geq 10$
\item $(3,11,m,m,m+1)$ for $m \geq 11$
\item $(3,11,m,m+2,m+2)$ for $m \geq 11$
\item $(3,12,m,m+1,m+2)$ for $m \geq 12$
\item $(3,12,m,m+3,m+3)$ for $m \geq 12$
\item $(3,13,m,m,m+2)$ for $m \geq 13$
\item $(3,13,m,m+2,m+3)$ for $m \geq 13$
\item $(3,13,m,m+4,m+4)$ for $m \geq 13$
\end{itemize}
(This is shown using arguments as in the proof of Theorem \ref{four var even}. We omit the details.)
\end{itemize}
Furthermore, computations using \cocoa \ suggest:

\begin{itemize}
\item[(iii)] When $a_1 = 3$ and $a_2 \leq 13$, all cases apart from the ones above have the WLP.
\item[(iv)] When $a_1 = 3,4$, some examples have the WLP and others do not.
\item[(v)] When $a_1 \geq 5$, $R/I$ fails the WLP.
\end{itemize}
\end{remark}

We now consider the case where $a_1+a_2+a_3+a_4$ odd.   There are
some interesting differences to Theorem \ref{four var even}.

\begin{theorem} \label{thm:four-var-odd}
Let $L$ be a general linear form and let $I = \langle
x_1^{a_1},x_2^{a_2},x_3^{a_3},x_4^{a_4},L^{a_5} \rangle$. Assume
that $a_1 + a_2 + a_3 + a_4$ is odd.  Let $\lambda =
\frac{a_1+a_2+a_3+a_4-5}{2}$.
\begin{itemize}
\item[(i)] If $a_5 \geq \lambda-1$ then the ring $R/I$ has the WLP.

\item[(ii)] If $a_5 < \lambda -1$ and $a_1+a_4 \ge  a_2+a_3$ then
$R/I$ fails the WLP from degree $\lambda-1$ to degree~$\lambda $.  These hypotheses force $a_1 \geq 5$.

\item[(iii)] If $a_5 < \lambda -1$, $a_1+a_4 <  a_2+a_3$ and
$ 2a_5+3-a_4-a_3-a_2+a_1\geq 0$ then $R/I$ fails the WLP from
degree $\lambda-1$ to degree $\lambda $. These hypotheses force $a_1 \geq 3$.
\end{itemize}
\end{theorem}

\begin{proof}
The outline of part of the proof is the same as that for Theorem \ref{four var even}, and we only highlight the differences.  First note that with the hypotheses of (ii), if $a_1 \leq 4$ then we have $a_5 < \lambda-1 = \frac{a_1+a_2 + a_3 + a_4 -7}{2} \leq \frac{2a_1 + 2a_4 - 7}{2} \leq a_4 + \frac{1}{2}$, a contradiction.  Similarly, with the hypotheses of (iii), if $a_1 = 2$ then we have $a_2+a_3+a_4 -5  \leq 2a_5 < 2\lambda-2 = a_2+a_3+a_4-5$, again a contradiction.

Now, for all three parts of the theorem, in the computation of how many powers of the quadratic polynomial $F_2$ contribute to the base locus, the result now is that it is $\lambda - a_5$.  Thus, in the first step we want to compute
\begin{equation} \label{first step}
\dim {\mathcal L}_2 (2a_5 - \lambda; a_5 - a_1 +1, a_5 - a_2+1, a_5-a_3+1,a_5-a_4+1,1).
\end{equation}
In the second step, we obtain that if $(a_5-a_i+1) + (a_5-a_j+1) > 2a_5 - \lambda$ then $L_{ij}$ appears as a factor
\[
(a_5-a_i+1) + (a_5-a_j+1) - (2a_5-\lambda) = \lambda -a_i-a_j+2
\]
times.  Thus we let
\[
A_{ij} = [\lambda+2-a_i-a_j]_+
\]
as before, and formally we have
\begin{equation} \label{both cases odd}
F_\lambda = F_2^{\lambda-a_5} \cdot \prod_{1 \leq i<j\leq 4}
L_{ij}^{A_{ij}} \cdot G, \ \ \ \hbox{ where } \deg G = 2a_5 -
\lambda  - \sum_{1 \leq i < j \leq 4} A_{ij}.
\end{equation}
Notice that
\begin{equation}\label{constraint2}
\hbox{if } 2a_5 - \lambda  < \sum_{1 \leq i < j \leq 4} A_{ij}
\hbox{ then }  \dim[\langle
x_1^{a_1},x_2^{a_2},x_3^{a_3},x_4^{a_4},L^{a_5},\ell \rangle
]_{\lambda} = 0.
\end{equation}

For (i), we want to show that whenever $a_5 \geq \lambda-1$, the multiplication $(\times \ell) : [R/I]_{t} \rightarrow [R/I]_{t+1}$ has maximal rank.  Let $J = \langle x_1^{a_1},x_2^{a_2},x_3^{a_3},x_4^{a_4}\rangle$.  Notice that we have $h_{R/J}(\lambda) = h_{R/J}(\lambda+1)$.  We consider several cases.

\begin{itemize}
\item $t \geq \lambda$.  We know that the multiplication on $R/J$ from degree $t$ to degree $t+1$ is surjective, by the result of Watanabe and Stanley.  Since $R/I$ is a quotient, the same holds for $R/I$.  This holds no matter what $a_5$ is.

\item $t \leq a_5 -2$.  Then $[R/I]{t} = [R/J]_{t}$ and $[R/I]_{t+1} = [R/J]_{t+1}$, so again the result follows trivially.  Notice that as a result of these first two cases, we are done if $a_5 \geq \lambda+1$.

\item $(t,a_5)  = ( \lambda-1, \lambda)$ or $(t,a_5) = (\lambda-2,\lambda-1)$.  We know that the multiplication for $R/J$ is injective in either of these cases, and that $\dim [R/I]_{t} = \dim [R/J]_{t }$ and $\dim [R/I]_{t+1} = \dim [R/J]_{t+1} -1$.  Then we argue exactly as in the case $a_5 = \lambda$ at the beginning of the proof of Theorem \ref{four var even}, using that the $\lambda$-th (resp. $(\lambda-1)$-st) powers of linear forms span $[R]_\lambda$ (resp. $[R]_{\lambda-1}$).  In particular, this completes the argument if $a_5 = \lambda$.

\item $t=\lambda -1$, $a_5 = \lambda-1$.  This case is the most subtle, and we now give  the argument for this case.
\end{itemize}

We claim that the multiplication $(\times \ell) : [R/I]_{\lambda -1} \rightarrow [R/I]_\lambda$ is injective.  To see this, we will show that $\dim [R/I]_{\lambda-1} < \dim [R/I]_\lambda$, and that $\dim [R/(I,\ell)]_\lambda = \dim [R/I]_\lambda - \dim [R/I]_{\lambda-1}$.

First we compute $\dim [R/(I,\ell)]_\lambda$.  The first step (\ref{first step}) now becomes
\[
\dim [R/(I,\ell)]_\lambda = \dim \mathcal L_2(\lambda-2; \lambda-a_1, \lambda-a_2, \lambda-a_3, \lambda-a_4,1) .
\]

We now consider three cases.  First, if $a_1 = a_2 = a_3 = a_4-1$ then (\ref{first step}) becomes
\begin{equation} \label{step 2a}
\dim [R/(I,\ell)]_\lambda = \dim \mathcal L_2(2a_4 - 6; a_4-3, a_4-3,a_4-3,a_4-4,1) = 2a_4 -6 = 2a_1 -4
\end{equation}
thanks to Lemma \ref{indep cond}.

We now assume that we are not in the first case.  We obtain
\[
\begin{array}{rclcrcl}
A_{1,2} & = & \displaystyle \left [\frac{-a_1-a_2+a_3+a_4-1}{2} \right ]_+ \geq 0 \quad
& A_{1,3} & = & \displaystyle \left [ \frac{-a_1 + a_2 - a_3 + a_4-1}{2} \right ]_+ \geq 0 \\ \\
A_{1,4} & = & \displaystyle \left [ \frac{-a_1 + a_2 + a_3 -a_4 -1}{2} \right ]_+  \quad
& A_{2,3} & = & \displaystyle \left [ \frac{a_1 - a_2 -a_2 + a_4-1}{2} \right ]_+ \\ \\
A_{2,4} & = & \displaystyle \left [ \frac{a_1 - a_2 + a_3 -a_4 -1}{2} \right ]_+ = 0 \quad
& A_{3,4} & = & \displaystyle \left [ \frac{a_1 + a_2 - a_3 - a_4-1}{2} \right ]_+ = 0.
\end{array}
\]
Our second case is $a_2+a_3 > a_1+a_4$.  Then $A_{1,2}$, $A_{1,3}$ and $A_{1,4}$ are possibly non-zero.  A calculation shows that we then must have
\begin{equation} \label{step 2b}
\dim [R/(I,\lambda)]_\lambda = \dim \mathcal L_2(2a_1-3; a_1-1, a_1-2, a_1-2, a_1-2,1)  = 2a_1-3.
\end{equation}

Our third case is $a_1+a_4 > a_2+a_3$.  Then $A_{1,2}$, $A_{1,3}$ and $A_{2,3}$ are possibly non-zero.  Another calculation shows that we must have
\begin{equation} \label{step 2c}
\begin{array}{rcl}
\dim[R/(I,\ell)]_\lambda & = & \dim \mathcal L_2(a_1+a_2+a_3-a_4-3; (\lambda-a_4+1)^3, \lambda-a_4,1) \\
& = & a_1+a_2+a_3-a_4-3.
\end{array}
\end{equation}

Now we  compute the expected dimension $\dim [R/I]_\lambda - \dim [R/I]_{\lambda-1}$ (recall $a_5 = \lambda-1$) by using Lemma \ref{lem:expected-Hilb}.

In the first case ($a_1 = a_2 = a_3 = a_4-1$) we use the Koszul complex and easily compute that the expected dimension is $2a_4-6 = 2a_1-4$, agreeing with (\ref{step 2a}).

In the second case and using the observations about which $A_{i,j}$ are positive, the Koszul resolution gives the expected dimension $2a_1-3$, agreeing with (\ref{step 2b}).

In the third case, again using the observations about the $A_{i,j}$ and the Koszul complex, we obtain the expected dimension $a_1+a_2+a_3-a_4-3$, agreeing with (\ref{step 2c}).  Thus when $a_5 = \lambda-1$, $R/I$ has the WLP, concluding the proof of (i).
\smallskip

Now, we prove (ii). Hence, we are assuming that $a_1 + a_4 \ge a_2 +
a_3$. In fact,  since $a_1+a_2+a_3+a_4$ is odd, we actually have
$a_1+a_4>a_2+a_3$. We obtain
\begin{equation} \label{pos neg2}
\begin{array}{rclcrcl}
a_1+a_4-a_2-a_3 & > & 0 \hbox{\hspace{1.5cm}} & a_2 +a_3 - a_1 - a_4 & < & 0 \\
a_2+a_4-a_1-a_3 & > & 0 \hbox{\hspace{1.5cm}} & a_1 +a_3 - a_2 - a_4 & < & 0 \\
a_3+a_4-a_1-a_2 & > & 0 \hbox{\hspace{1.5cm}} & a_1 +a_2 - a_3 -
a_4 & < & 0.
\end{array}
\end{equation}
Therefore, we get
\[
\begin{array}{rcl}
a_5 - a_4 +1 - \sum_{i=1,2,3} [\lambda+2-a_i-a_4]_+  & = & a_5-a_4+1 \\ \\
a_5 - a_3 +1 - \sum_{i=1,2,4} [\lambda+2-a_i-a_3]_+  & = & a_5 - a_4 +2 \\ \\
a_5 - a_2 +1 - \sum_{i=1,3,4} [\lambda+2-a_i-a_2]_+  & = & a_5 - a_4 +2\\ \\
a_5 - a_1 +1 - \sum_{i=2,3,4} [\lambda+2-a_i-a_1]_+  & = & a_5 - a_4 +2 \\ \\
\end{array}
\]
and
\[
2a_5 - \lambda  - \sum_{1 \leq i < j \leq 4} A_{ij} = 2(a_5-a_4+2).
\]
Concluding, we want the dimension of the linear system
$${\mathcal L_2}(2a_5-2a_4+4;a_5-a_4+2,a_5-a_4+2,a_5-a_4+2,a_5-a_4+1,1).$$
By Lemma \ref{indep cond} we  have
\begin{eqnarray*}
\lefteqn{
\dim_k {\mathcal L_2}(2a_5-2a_4+4; (a_5-a_4+2)^3,a_5-a_4+1,1) } \\
& \ge & \dim_k {\mathcal L_2}(2a_5-2a_4+4; (a_5-a_4+2)^3,a_5-a_4+1)-1 \\
& = & \dim_k {\mathcal L_2}(2a_5-2a_4+4; (a_5-a_4+2)^4)+a_5-a_4+1 \\
& = & 2a_5-2a_4+4.
\end{eqnarray*}
This is clearly positive.  Arguing as in the proof of Theorem \ref{four var even}, and using the inequalities (\ref{pos neg2}), we compute the expected dimension and we get
{\small
\[
\begin{array}{l}
[\dim [R/I]_{\lambda }- \dim [R/I]_{\lambda-1]}]_+  =   \left [
\binom{\lambda +2 }{2} - \sum_{i=1}^4 \binom{\lambda-a_i +2}{2} +
\sum_{1\le i <j \le 3} \binom{\lambda-a_j - a_i +2}{2} -
\binom{\lambda-a_5+2}{2} \right ]_+ .
\end{array}
\]
}\noindent If the part inside the brackets is negative then clearly
the actual dimension exceeds the expected one, and we are done.  If
not, a straightforward computation shows that
\begin{eqnarray*}
\lefteqn{
\dim_k [R/(I,\ell)]_{\lambda }-[\dim_k [R/I]_{\lambda } - \dim_k [R/I]_{\lambda-1}] } \\
& \ge &  \left ( 2a_5 - 2a_4 +4 \right )  \\
&&      - \left (  \binom{\lambda+2}{2} - \sum_{i=1}^4
 \binom{\lambda-a_i +2}{2} + \sum_{1\le i < j \le 3} \binom{\lambda-a_j - a_i +2}{2}  - \binom{\lambda-a_5+2}{2}  \right )    \\
& = & \binom{\lambda -a_5}{ 2}  > 0.
\end{eqnarray*}
Thus the actual dimension exceeds the expected dimension,  and this completes the proof of (ii).
\smallskip

(iii)   We break this into two cases:  first we will assume that $a_1+a_4 +3\le a_2+a_3$, and then we will handle the case $a_1+a_4 +1 = a_2+a_3$.

So to begin, our assumptions now are that $a_1+a_2+a_3+a_4$
is odd, $a_5 < \lambda -1$,  $a_1+a_4 +3\le a_2+a_3$ and
$ 2a_5+3-a_4-a_3-a_2+a_1\geq 0$. Hence, we have
\begin{equation} \label{plus minus (iv)}
\begin{array}{rclcrcl}
a_3+a_4-a_1-a_2 & > & 0 \hbox{\hspace{1.5cm}} & a_1 +a_2 - a_3 - a_4 & < & 0 \\
a_2+a_4-a_1-a_3 & > & 0 \hbox{\hspace{1.5cm}} & a_1 +a_3 - a_2 - a_4 & < & 0 \\
a_2+a_3-a_1-a_4 & > & 0 \hbox{\hspace{1.5cm}} & a_1 +a_4 - a_2 -
a_3 & < & 0 .
\end{array}
\end{equation}
Now we examine the linear system that remains when we remove the
one-dimensional base locus. Observe that
\[
\begin{array}{rcl}
a_5 - a_4 +1 - \sum_{i=1,2,3} [\lambda+2-a_i-a_4]_+  & = & a_5+1-  \frac{a_2+a_3+a_4-a_1-1}{2} \\ \\
a_5 - a_3 +1 - \sum_{i=1,2,4} [\lambda+2-a_i-a_3]_+  & = & a_5 +1-  \frac{a_2+a_3+a_4-a_1-1}{2} \\ \\
a_5 - a_2 +1 - \sum_{i=1,3,4} [\lambda+2-a_i-a_2]_+  & = & a_5 +1-  \frac{a_2+a_3+a_4-a_1-1}{2}\\ \\
a_5 - a_1 +1 - \sum_{i=2,3,4} [\lambda+2-a_i-a_1]_+  & = & a_5 +2-  \frac{a_2+a_3+a_4-a_1-1}{2} \\ \\
\end{array}
\]
and
\[
2a_5 - \lambda  - \sum_{1 \leq i < j \leq 4} A_{ij} =
2a_5+4-a_4-a_3-a_2+a_1 =: 2b.
\]
The additional hypothesis $ 2a_5+3-a_4-a_3-a_2+a_1\geq 0$  guarantees that the orders of vanishing are  $\ge 0$. Therefore,
applying Lemma \ref{indep cond}, we obtain
\begin{eqnarray*}
\dim_k  [R/(I,\ell)]_{\lambda }
& = & \hbox{ $ \dim_k {\mathcal L_2}(2b; b, (a_5+1-
 \frac{a_2+a_3+a_4-a_1-1}{2})^3,1) $} \\
& \ge &  2a_5+5-(a_2+a_3+a_4-a_1) -1.
\end{eqnarray*}

Notice that the hypothesis also guarantees that this value is $>0$.
The computation of the ``expected" dimension is very similar to what
we did above. The extra hypothesis $a_1+a_4+3\le a_2+ a_3$ implies
that only  $\lambda - a_1 - a_2 +2$, $\lambda - a_1 - a_3+2$ and
$\lambda - a_1 - a_4+2$ are $> 0$.  We again have $a_4 < a_1 + a_2 +
a_3$.   Thus, Lemma \ref{lem:expected-Hilb} provides
\begin{eqnarray}
  \label{actual value}
\lefteqn{
\big [\dim_k [R/I]_\lambda - \dim_k [R/I]_{\lambda-1]} \big ]_+  } \nonumber \\
& = & \big [h_A(\lambda) - h_A (\lambda-a_5) - h_A(\lambda-1) +
     h_A(\lambda-a_5-1) \big ]_+ \nonumber \\
& = &  \left [ \binom{\lambda+2}{2} - \sum_{i=1}^4 \binom{\lambda-a_i +2}{2} +
 \sum_{i=2}^4 \binom{\lambda-a_1 - a_i +2}{2}  - \binom{\lambda-a_5+2}{2}   \right
 ]_+. \nonumber
\end{eqnarray}
If this number is zero, then we are done. Otherwise,  a
straightforward computation shows that
\begin{eqnarray*}
\lefteqn{
\dim_k [R/(I,\ell)]_{\lambda }-[\dim_k [R/I]_{\lambda } -
\dim_k [R/I]_{\lambda-1}] } \\
& = &  2a_5+4-(a_2+a_3+a_4-a_1) \\
& & - \left (  \binom{\lambda+2}{2} -
\sum_{i=1}^4 \binom{\lambda-a_i +2}{2} +
\sum_{i=2}^4 \binom{\lambda-a_1 - a_i +2}{2}  -
\binom{\lambda-a_5+2}{2}  \right ) \\
& = & \binom{\lambda -a_5}{2} > 0.
\end{eqnarray*}
Thus in either case the actual value of $\dim [R/(I,\ell)]_\lambda$
exceeds the expected dimension.
This
completes the proof of the case $a_1+a_4+3  \leq a_2+a_3$.

Finally, we assume that $a_1+a_4+1 = a_2+a_3$.  We note first that
this assumption actually forces the stronger condition $a_1 \geq 4$,
since $a_4 \leq a_5 < \lambda-1$ implies $7 < a_1+a_2+a_3-a_4 =
2a_1+1$, and hence $a_1 > 3$. The condition   $a_1+a_4+1 = a_2+a_3$
implies that $\lambda-a_1-a_4 +2 = 0$, so the computation of the
expected dimension above can only get smaller, while the computation
of $\dim [R/(I,\ell)]_\lambda$ remains unchanged.  Thus the
difference can only grow, and we again have shown the failure of the
WLP.  This completes the proof of (iii).
\end{proof}

In the previous results we excluded the case $a_1 = 2$ for the most
part.  The reason is that then the algebra does have the WLP.

\begin{theorem}
  \label{thm:wlp-four-variables}
Let $L$ be a general linear form in $R := k[x_1,\ldots,x_4]$,  and
let  $I$ be the ideal $\langle
x_1^{2},x_2^{a_2},x_3^{a_3},x_4^{a_4},L^{a_5} \rangle$ of $R$.  Then
the algebra $R/I$ has the WLP.
\end{theorem}

The proof will be based on a result about almost complete
intersections  in three variables generated by powers of four
general linear forms. According to \cite{ss} such an algebra has the
WLP, i.e., multiplication by a general linear form has maximal rank.
We show that this is also true when one multiplies by the square of
such a form.

\begin{proposition}
  \label{prop:max-rank-three-var}
Let $\ell_2,\ldots,\ell_5, \ell$ be five general linear forms of  $S
:= k[x,y,z]$, and let $\fa \subset S$ be the ideal $\langle
\ell_2^{a_2},\ldots,\ell_5^{a_5}\rangle$. Set $A := S/\fa$. Then,
for each integer $j$, the multiplication map $\times \ell^2:
[A]_{j-2} \to [A]_j$ has maximal rank.
\end{proposition}

\begin{proof}
If any of the numbers $a_2,\ldots,a_5$ equals one, then $A$ has the
SLP by \cite{HMNW}, so the claim is true. Thus, we may assume $2 \le
a_2 \le \cdots \le a_5$.

First, assume also that
\[
a_5 \le \frac{a_2 + a_3 + a_4 - 3}{2}.
\]
Define integers $p$ and $b$ by
\[
p := \left \lfloor \frac{a_2 + a_3 + a_4 + a_5 - 4}{3}
 \right \rfloor \quad \text{and} \quad  a_2 + a_3 + a_4 + a_5 =
 3 (p+1) + b,
\]
thus $1 \le b \le 3$.

According to \cite{ss}, $A$ has the WLP and $A/\ell A$ has socle
degree $p$. This implies that multiplication by $\ell$ on $A$ is
injective until degree $p$ and  surjective in larger degrees.
Symbolically, this reads as
\begin{equation}
  \label{eq:maps}
[A]_0 \hookrightarrow \cdots \hookrightarrow [A]_{p-1}
\hookrightarrow [A]_p \twoheadrightarrow [A]_{p+1}
\twoheadrightarrow \cdots.
\end{equation}
Hence, to show our claim it suffices to prove that the
multiplication $\times \ell^2: [A]_{p-1} \to [A]_{p+1}$ has maximal
rank, which is equivalent to
\begin{equation}
  \label{eq:max-rank}
\dim_k [A/\ell^2 A]_{p+1} = \max \{0, h_A (p+1) - h_A (p-1)\}.
\end{equation}

In order to see this, we first compute the left-hand side and then
the right-hand side. Using Theorem \ref{thm:inverse-system}, we get
\begin{eqnarray*}
  \dim_k [A/\ell^2 A]_{p+1}  & = & \dim_k [S/\langle \ell^2, \ell_2^{a_2},\ldots,\ell_5^{a_5} \rangle]_{p+1} \\
  &  = &  \dim_k [\fp^{p} \cap \fp_2^{p - a_2 +2} \cap \ldots \cap \fp_5^{p - a_5 +2}]_{p+1} \\
  & = & \dim_k \cL_2 (p+1; \  p,  p - a_2 +2,\ldots,p - a_5 +2),
\end{eqnarray*}
where $\fp, \fp_2,\ldots,\fp_5$ are the homogenous ideals of five
general points in $\PP^2$. Let $Q \in S$ be the unique quadric that
vanishes at these five points. Again, we use Bezout's theorem to
estimate the multiplicity of $Q$ in the base locus of the linear
system  $\cL_2 (p+1; \ p, p - a_2 +2,\ldots, p - a_5 +2)$. For an integer
$j$, the condition
\begin{eqnarray*}
  2 (p+1 - 2j) & \ge & p - j + \sum_{i=2}^5 [p- a_i + 2 - j] \\
  & = & 5 p + 8 - 5 j -   \sum_{i=2}^5 a_i \\
  & = & 5p + 8 - 5 j - [3 (p+1) + b] \\
  & = & 2 p + 5 - 5 j - b
\end{eqnarray*}
is equivalent to $j \ge 3 - b$. It follows that $Q$ appears with
multiplicity at least $3 - b$ in the base locus. Thus, we get
\begin{eqnarray*}
  \dim_k [A/\ell^2 A]_{p+1}  & = &
          \dim_k \cL_2 (p+1; \ p, p - a_2 +2,\ldots,p - a_5 +2 ) \\
   & = &  \dim_k \cL_2 (p+2 b - 5;\  p + b - 3, p + b - a_2 -1,\ldots, p + b - a_5 - 1).
\end{eqnarray*}
The latter linear system is clearly empty if $b = 1$. To compute its
dimension if $2 \le b \le 3$, we consider the lines $L_{i}$, $2 \le
i \le 5$, passing  through the points $\fp$ and $\fp_i$. By Bezout's
theorem, the line $L_i$ appears with multiplicity at least $B_i$ in
the base locus, where
\begin{eqnarray*}
  B_i & := & \left [ p+b - 3 + p + b - a_i - 1 - \{p + 2 b -5\} \right ]_+ \\
  & = & \left [p+1 - a_i \right ]_+.
\end{eqnarray*}
Our assumption $a_5 \le \frac{a_2 + a_3 + a_4 - 3}{2}$ implies $a_5
\leq p+1$. Thus, we get $B_i = p + 1 - a_i$, so
\[
B_2 + \cdots + B_5 = 4 p + 4 - [a_2 + a_3 + a_4 + a_5] = 4 p + 4 - [3 (p+1) + b] = p + 1 - b.
\]
Removing the lines from the base locus we obtain
\begin{eqnarray*}
  \dim_k [A/\ell^2 A]_{p+1}  & = &
    \dim_k \cL_2 (p+2 b - 5;\ p + b - 3, p + b - a_2 -1,\ldots, p + b - a_5 - 1)\\
  & = & \dim_k \cL_2 (3 b - 6;\ 2 b - 4, b-2,\ldots,b-2).
\end{eqnarray*}
It follows that $\dim_k [A/\ell^2 A]_{p+1} = 1$ if $b = 2$.

If $b =3$, then we get, using again Theorem
\ref{thm:inverse-system},
\begin{eqnarray*}
\dim_k [A/\ell^2 A]_{p+1}  & = &
  \dim_k \cL_2 (3 ;\ 2 , 1,\ldots,1) \\
& = & \dim_k [S/\langle \ell^2,\ell_2^3,\ldots,\ell_5^3\rangle]_3 \\
& = & \dim_k [S/\langle \ell^2 \rangle]_3 - 4 \\
& = & 3
\end{eqnarray*}
because the linear forms are general. Summarizing, we have shown so
far that
\begin{eqnarray}
\label{eqn:hilb-func}
\dim_k [A/\ell^2 A]_{p+1}  & = & \left \{
\begin{array}{cl}
  0 & \text{if } b = 1 \\
  1 & \text{if } b = 2 \\
  3 & \text{if } b = 3. \\
\end{array}
\right.
\end{eqnarray}

Now we compute the right-hand side of Equation \eqref{eq:max-rank}.
To this end consider the ring $B := S/\langle \ell_3^{a_3},
\ell_4^{a_4},  \ell_5^{a_5} \rangle$. Observe that
\begin{eqnarray*}
  p+1 - a_3 - a_4  & = & \left \lfloor \frac{a_2 + a_3 + a_4 + a_5 - 1}{3} \right  \rfloor
                         - a_3 - a_4\\
      & \le & \frac{a_2 +  a_5 - 1 - 2 a_3 - 2 a_4}{3} \\
      & \le & \frac{3 a_2 - 5 - 3 a_3 - 3 a_4}{6} \\
      & < & 0,
\end{eqnarray*}
where we used again $a_5 \le \frac{a_2 + a_3 + a_4 - 3}{2}$. Hence,
the Koszul resolution of the complete intersection $B$ provides for
its Hilbert function if $j \le p+1$ that
\begin{equation}
 \label{eq:hilb-ci}
  h_B (j) = \binom{j+2}{2} - \sum_{i=3}^5 \binom{j+2 - a_i}{2},
\end{equation}
where, as above, we define a binomial coefficient $\binom{c}{2}$ to
be zero if $c < 0$. Since the complete intersection $B$ has the SLP
and $A \cong B/\ell_2^{a_2} B$, we get
\[
h_A (p+1) - h_A (p-1) = \left [h_B (p+1) - h_B (p+1 - a_2)  \right ]_+
- \left [h_B (p-1) - h_B (p-1 - a_2) \right ]_+.
\]
One easily checks that our assumptions provide $p+1 \le \left \lceil
\frac{a_3+a_4+a_5-3}{2} \right \rceil$. Since $a_3+a_4+a_5-3$ is the
socle degree of $B$, this implies that $h_B (j-1) \le h_B (j)$ if $j
\le p+1$. Hence the last formula simplifies to
\[
h_A (p+1) - h_A (p-1) = h_B (p+1) - h_B (p+1 - a_2)
- h_B (p-1) + h_B (p-1 - a_2).
\]
The socle degree of $A/\ell A$ is at most the socle degree of
$S/\langle \ell, \ell_2^{a_2},\ell_3^{a_3}\rangle$, thus $p \le a_2
+ a_3 -2$. Combined with Equation \eqref{eq:hilb-ci} and using
$\binom{j+1}{2} - \binom{j-1}{2}$ if $j \ge 1$, this provides
\begin{eqnarray*}
  h_A (p+1) - h_A (p-1) & = & \binom{p+3}{2} - \binom{p+1}{2} -
     \sum_{i=3}^5 \left [ \binom{p+3-a_a}{2} - \binom{p+1-a_i}{2} \right ]  \\
     & &  -
     \left [ \binom{p+3-a_2}{2} - \binom{p+1-a_2}{2} \right ] \\
  & = & - 6 p - 9 + 2 \sum_{i=2}^4 a_i \\
  & = & 2b - 3.
\end{eqnarray*}
Hence, we get
\[
\left [h_A (p+1) - h_A (p-1)  \right ]_+ = \left [  2b-3 \right ]_+ =
\left \{
\begin{array}{cl}
  0 & \text{if } b = 1 \\
  1 & \text{if } b = 2 \\
  3 & \text{if } b = 3. \\
\end{array}
\right.
\]
Comparing with Equation \eqref{eqn:hilb-func}, this establishes the
desired equality \eqref{eq:max-rank}.
\smallskip

It remains to consider the case $a_5 > \frac{a_2 + a_3 + a_4 -
3}{2}$.  Let us call  $\fb:= \langle
\ell_2^{a_2},\ell_3^{a_3},\ell_4^{a_4}\rangle$ and set  $B := S/\fb$. Note that the socle degree of $B$ is $a_2 + a_3 + a_4 -
3$.
 Stanley \cite{stanley} and  Watanabe \cite{watanabe} showed that, in characteristic 0, SLP  holds for an artinian complete intersection generated by powers of linear forms.
  In particular,
for each integer $j$, the multiplication map $\times \ell^2:
[B]_{j-2} \to [B]_j$ has maximal rank. We want to prove that, for each integer $j$, the multiplication map $\times \ell^2:
[A]_{j-2} \to [A]_j$ has maximal rank. To this end, we will distinguish several cases:

(a) Assume $a_2+a_3+a_4$ odd.

(a1) For any $j< a_5$, we have $[A]_j\cong [B]_j$ and hence  $\times \ell^2:
[A]_{j-2} \to [A]_j$ has maximal rank.

(a2) For any $j\ge a_5+1$ or $j=a_5$ and $a_5 > \frac{a_2 + a_3 + a_4 -
3}{2}+1$, we have $j-2\ge \frac{a_2 + a_3 + a_4 -
3}{2}$. Hence, the surjectivity $\times \ell^2:
[B]_{j-2}\twoheadrightarrow  [B]_j$  together with the commutative diagram

 \[
 \begin{array}{cccccccccc}
 [B]_{j-2} & \twoheadrightarrow & [A]_{j-2} \\
 \downarrow && \downarrow \\

 [B]_{j} & \twoheadrightarrow & [A]_{j}
  \end{array}
  \]
allows us to conclude that $\times \ell^2:
[A]_{j-2} \to [A]_j$ has maximal rank.

(a3) For $j=a_5=  \frac{a_2 + a_3 + a_4 -
3}{2}+1$ we have $[A]_{j-2}\cong [B]_{j-2} \cong [B]_j \twoheadrightarrow [A]_j$. Therefore, we also conclude that  $\times \ell^2:
[A]_{j-2} \to [A]_j$ has maximal rank.

(b) Assume $a_2+a_3+a_4$ even.

(b1) For any $j< a_5$, we have $[A]_j\cong [B]_j$ and hence  $\times \ell^2:
[A]_{j-2} \to [A]_j$ has maximal rank.

(b2)  It is identical to (a2).

(b3) For $j=a_5=  \frac{a_2 + a_3 + a_4 -
2}{2}$ we have to prove that $\times \ell^2:
[A]_{j-2} \to [A]_j$ is injective. Since $\dim [B]_{j-2} < \dim [B]_j$, the image of $[B]_{j-2}$ in $[B]_j$ under multiplication by $\ell ^2$ is not surjective.  Hence the vector space $[\langle \ell _2^{a_2},\ell _3^{a_3},\ell _4^{a_4}, \ell ^2 \rangle]_j$ is not all of $S_j$.  But in characteristic zero the $j$-th powers of linear forms span $[S]_j$.  Thus for a general $L$, the image of $L^j=L^{a_5}$ in $[B]_j$ is outside the image of $[B]_{j-2}$ in $[B]_j$.  Thus $(\times \ell^2) : [A]_{j-2} \rightarrow [A]_j$ is injective, since it is for $B$ in that degree.  This completes the proof of the proposition.


\end{proof}

\begin{remark}
  \label{rem:mult-higher-powers}
Observe that extensions of Proposition \ref{prop:max-rank-three-var}   to multiplications by higher powers of $\ell$ fail in general. There are many such examples.  The smallest is when $a_2 = \cdots = a_5 = 3$, for which multiplication by $\ell^3$ fails to have maximal rank from degree 1 to degree 4. It is easy to see that this extends to the case $a_2 = \cdots = a_5 = d$, for which multiplication by $\ell^d$ fails to have maximal rank from degree $d-2$ to degree $2d-2$.  Many more complicated examples (produced by \cocoa) exist as well.
\end{remark}

We are ready for the proof of Theorem \ref{thm:wlp-four-variables}.

\begin{proof}[Proof of Theorem \ref{thm:wlp-four-variables}]
If one of the numbers $a_2,\ldots,a_5$ is one, then the result
follows  by \cite{ss}. Thus, we may assume $2 \le a_2 \le \cdots \le
a_5$.

If $a_5 > \frac{a_2 + a_3 + a_4 - 3}{2}$, then $R/I$ has the WLP by
Theorems \ref{four var even}(i) and \ref{thm:four-var-odd}(i). Thus,
we may assume $a_5 \le \frac{a_2 + a_3 + a_4 - 3}{2}$ for the
remainder of the proof. Let $\ell \in R$ be another general linear
form. We have to show that for all integers $j$,
\begin{equation}
  \label{eq:wlp-condition}
\dim_k [R/ (I, \ell)]_j = \max \{0, h_{R/I} (j) - h_{R/I} (j-1)\}.
\end{equation}

To this end, consider the ideal $J := \langle
x_2^{a_2},x_3^{a_3},x_4^{a_4},L^{a_5} \rangle \subset R$. The
complete intersection $R/J$ has the SLP, which implies
\[
\left [ h_{R/I} (j) - h_{R/I} (j-1) \right ]_+ = \left [
\left [ h_{R/J} (j) - h_{R/J} (j-2)  \right ]_+
 - \left [ h_{R/J} (j-1) - h_{R/J} (j-3) \right ]_+
\right ]_+.
\]
The complete intersection $R/J$ has socle degree $a_2+\cdots+a_5 - 4
\ge 4$, thus the multiplication $\times \ell: [R/J]_{j-1} \to
[R/J]_j$ is injective if $j \le \left \lceil \frac{a_2 + a_3 + a_4 +
a_5 - 4}{2} \right \rceil$. One checks that
\[
\left \lceil \frac{a_2 + a_3 + a_4 + a_5 -
4}{2} \right \rceil \ge \left \lfloor \frac{a_2 + a_3 + a_4 + a_5 - 1}{3} \right
\rfloor =: p + 1.
\]
It follows that $h_{R/J} (j-1) \le h_{R/J} (j)$ if $j \le p+1$.
Thus, if $j \le p+1$, then we get that
\begin{eqnarray*}
  h_{R/I} (j) - h_{R/I} (j-1) & = & h_{R/J} (j) - h_{R/J} (j-2)
    - h_{R/J} (j-1) +  h_{R/J} (j-3) \\
  & = & \left [ h_{R/J} (j) - h_{R/J} (j-1)  \right ]_+ -
        \left [ h_{R/J} (j-2) - h_{R/J} (j-3) \right ]_+ \\
  & = & h_{R/(J, \ell)} (j) -  h_{R/(J, \ell)} (j-2),
\end{eqnarray*}
where we used again that $R/J$ has the SLP. Invoking Proposition
\ref{prop:max-rank-three-var}, we obtain
\[
\left [ h_{R/I} (j) - h_{R/I} (j-1)  \right ]_+ = h_{R/(I, \ell)} (j)
\]
if $ j \le p+1$. This equality is also true (using the same
computations) if $h_{R/J} (p+1) \leq h_{R/J} (p+2)$. Otherwise, we
get $\left [ h_{R/I} (p+2) - h_{R/I} (p+1)  \right ]_+ = 0$.
However, using \eqref{eq:maps}, we get $[R/(I, \ell)]_{p+2} = 0$.
Hence, we have in any case
\[
\left [ h_{R/I} (p+2) - h_{R/I} (p+1)  \right ]_+ =
h_{R/(I, \ell)} (p+2) = 0,
\]
which completes the argument.
\end{proof}


\section{Almost uniform powers of linear forms in 5 variables} \label{5-variables}

In this section, we let $R = k[x_1,x_2,x_3,x_4,x_5]$, where $k$ is a field of characteristic zero and we will apply the approach described in \S 2 and results on fat points in $\PP^3$ to determine exactly when an ideal generated by uniform  powers of six general linear forms in $R$ fails the WLP.  Some non-uniform cases are also discussed.  To this end, without loss of generality we set $I = \langle x_1^{a_1},x_2^{a_2},x_3^{a_3},x_4^{a_4},x_5^{a_5},L^{a_6} \rangle$, where $L$ is a general linear form and $a_1 \leq a_2 \leq a_3 \leq a_4 \leq a_5\le a_6$.

\begin{theorem} \label{five var uniform case}
Let $L$ be a general linear form and let $I = \langle x_1^{d},x_2^{d},x_3^{d},x_4^{d},x_5^d,L^{d} \rangle$.
 Then, the ring $R/I$ fails the WLP if and only if $d>3$.
\end{theorem}

\begin{proof} Using \cocoa\ we check that if $d<4$ then $R/I$  has the WLP. Assume $d\ge 4$ and we will show that $R/I$ fails the WLP in degree $2d-1$. To this end, we take  $\ell \in R$  a general linear form. As mentioned in Section \ref{gen appr}, it is enough to show that
\begin{equation}\label{to show2}
\dim [R/(I,\ell)]_{2d-1} >  [\dim[R/I]_{2d-1} - \dim[R/I]_{2d-2}]_+.
\end{equation}
We will compute the left-hand side using the approach of Section \ref{gen appr} and the right-hand side using the fact that $R/J$ has the SLP.

Now let $S = k[x_1,x_2,x_3,x_4]$.  We want to compute the vector space dimension
\begin{equation}\label{desired dim2}
\dim_k [\fq _1^{d} \cap \fq _2^{d} \cap \fq _3^{d} \cap \fq _4^d \cap \fq_5^{d} \cap \fq_6^d ]_{2d-1}
= \dim_k {\mathcal L}_3(2d-1; d^6).
\end{equation}
 Applying a sequence of cubo-cubic Cremona transformations $$\begin{array}{ccl} \PP^3 & \dashrightarrow & \PP^3 \\ (x_1:x_2:x_3:x_4) & \dashrightarrow & (x_1^{-1},x_2^{-1},x_3^{-1},x_4^{-1})=(x_2x_3x_4:x_1x_3x_4:x_1x_2x_4:x_1x_2x_3)
\end{array}$$
we will transform the  last linear system to another one which has the same dimension, but it will be non-special and hence we will be able to compute its dimension. In fact, we apply Lemma  \ref{lem:Cremona} and we get

\begin{equation}\label{cremona}
\begin{array}{l}
\dim {\mathcal L_3}(2d-1; d^6)=\dim {\mathcal L_3}(2d-3; d^2,(d-2)^4) \\ \\
= \dim {\mathcal L_3}(2d-5; (d-2)^4,(d-4)^2) \\ \\ = \dim {\mathcal L_3}(2d-7; (d-4)^6).
\end{array}
\end{equation}
Since $2(2d-7)\ge 4(d-4)$, the linear system ${\mathcal L_3}(2d-7; (d-4)^6)$ is in standard form. Therefore, Theorem  \ref{thm:fat-points-P3} provides that  it is non-special and its dimension is given by 
\begin{equation} \label{barcelona}
 \dim {\mathcal L_3}(2d-7; (d-4)^6)={2d-4\choose 3 } - 6{d-2\choose 3}.
 \end{equation}

Now we have to compute the right-hand side of (\ref{to show2}).  Let $A = R/J$, where $J = \langle x_1^{d},x_2^{d},x_3^{d},x_4^{d},x_5^d \rangle$.
Since $R/J$ has the SLP, we have $$0 \leq h_{R/I}(2d-1) = h_A(2d-1) - h_A(d-1) .$$
Therefore,
\[
\begin{array}{l}
[\dim [R/I]_{2d-1} - \dim [R/I]_{2d-2}] \\ \\ =  h_A(2d-1) - h_A (d-1) - h_A(2d-2) + h_A(d-2)\\ \\
=  {2d+2\choose 3}-6{d+2\choose 3}.
\end{array}
\]
We easily verify that $$\dim [R/(I,\ell)]_{2d-1}-[\dim [R/I]_{2d-1} - \dim [R/I]_{2d-2}]=4$$ and this shows that $R/I$ fails the WLP in degree $2d-1$, which is what we wanted to prove.
\end{proof}

\begin{remark} 
Note that  Equation \eqref{barcelona} could also have been proven using Proposition 3.5 of \cite{ceg}.  However, we use our approach because it also applies to set-ups below where the hypotheses of the latter proposition are not satisfied.
\end{remark}

There  are several possible extensions of the above  theorem.  First, we can ask whether the WLP  property holds for the case of non-uniform powers and, in particular, we can ask what happens in the almost uniform case. We have

\begin{theorem} \label{five var almost-uniform case}
Let $L$ be a general linear form and let $I = \langle x_1^{d},x_2^{d},x_3^{d},x_4^{d},x_5^d,L^{d+e} \rangle$ with $e\ge 1$.  Then:

\begin{itemize}
\item[(i)] If $d$ is odd, then $R/I$ has the WLP if and only if $e \ge
\frac{3d-5}{2}$.

\item[(ii)] If $d$ is even, $R/I$ has the WLP if and only if $e \ge
\frac{3d-8}{2}$.
\end{itemize}

\end{theorem}

\begin{proof}

Set $\lambda :=\lfloor \frac{5d-5}{2} \rfloor $.  We will actually prove the following sequence of statements:

\begin{itemize}
\item[(i$'$)]   If $d+e\ge \lambda $, then $R/I$ has the WLP.
\item[(ii$'$)] If $d$ is odd and $e\le d-2$, then $R/I$ fails the WLP.
\item[(iii$'$)] If $d$ is even and $e\le d-3$, then $R/I$ fails the WLP.
\item[(iv$'$)] If $d$ is odd and $d-1 \leq e \leq \frac{3d-7}{2}$ then $R/I$ fails the WLP.
\item[(v$'$)] If $d$ is even and $d-2 \leq e \leq \frac{3d-10}{2}$ then $R/I$ fails the WLP.
\item[(vi$'$)] If $d$ is even and $e = \frac{3d-8}{2}$ (i.e. $d+e=\lambda -1$) then $R/I$ has the WLP.
\end{itemize}

Throughout this proof we will denote $A = R/J$, where $J = \langle x_1^d,x_2^d,x_3^d,x_4^d,x_5^d \rangle$.

For (i$'$) the proof is the same as in Theorems 3.2(i) and 3.5(i).

(ii$'$) Since $d+e\le 2d-2$, we can write $e=2j+\epsilon $ with $0\le \epsilon \le 1$ and $j\le \frac{d-3}{2}$. We will show that $R/I$ fails the WLP in degree $2d-1+j$. To this end, we take  $\ell \in R$,   a general linear form. As mentioned in Section \ref{gen appr}, it is enough to show that
\begin{equation}\label{to show3}
\dim [R/(I,\ell)]_{2d-1+j} >  [\dim[R/I]_{2d-1+j} - \dim[R/I]_{2d-2+j}]_+.
\end{equation}
We will compute the left-hand side using the approach of Section \ref{gen appr} and the right-hand side using Lemma \ref{lem:expected-Hilb}.

We begin by computing the vector space dimension
\begin{eqnarray*}\label{desireddim2}
\lefteqn{
\dim_k [\fq _1^{d+j} \cap \fq _2^{d+j} \cap \fq _3^{d+j} \cap \fq _4^{d+j} \cap \fq_5^{d+j} \cap \fq_6^{d+j-e} ]_{2d-1+j} \hspace*{3cm} }\\
 & = & \dim_k {\mathcal L}_3(2d-1+j; (d+j)^5,d+j-e).
\end{eqnarray*}
Applying a sequence of cubo-cubic Cremona transformations (see Lemma \ref{lem:Cremona}),   we get
\begin{eqnarray*}\label{cremona2}
\lefteqn{
\dim_k {\mathcal L_3}(2d-1+j;\ (d+j)^5,d+j-e) } \\
& = &  \dim_k {\mathcal L_3}(2d-3-j;\ d+j, d+j-e), (d-j-2)^4  \\
& = & \dim_k {\mathcal L_3}(2d-5-3j+e;\ d-j-2+e, (d-j-2)^3, (d-3j-4+e)^2) \\
& = &  \dim_k {\mathcal L_3}(2d-7-5j+2e;\ d-3j-4+2e, (d-3j-4+e)^5).
\end{eqnarray*}
(Here we use the hypothesis $d+e\le 2d-2$ to guarantee that
$d-3j-4+2e \ge 0$.)

Since $2(2d-7-5j+2e)\ge 3(d-3j-4+e)+d-3j-4+2e$, the linear system ${\mathcal L_3}(2d-7-5j+2e; d-3j-4+2e, (d-3j-4+e)^5)$ is in standard form. Therefore, Theorem \ref{thm:fat-points-P3} provides
\begin{eqnarray*}
\lefteqn{
\dim_k {\mathcal L_3}(2d-7-5j+2e; (d-3j-4+e)^5,d-3j-4+2e) } \\[1ex]
& = &
 \left [ {2d-4-5j+2e\choose 3 } - 5{d-2-3j+e\choose 3}-
 {d-2-3j+2e\choose 3} \right ]_+ + 5 \cdot \binom{e-j}{3}.
\end{eqnarray*}
We claim that
\begin{eqnarray*}
\lefteqn{
{2d-4-5j+2e\choose 3 } - 5{d-2-3j+e\choose 3}-
 {d-2-3j+2e\choose 3}  + 5 \cdot \binom{e-j}{3} } \\
   & =  &   \frac{1}{6} (d-2-2j)
[(2d^2-3e^2+12e(1+j)-4(3+8j+4j^2)+d(3e-2-2j)]
\end{eqnarray*}
is positive. In fact, using $d \ge 2j+3$ and $e \ge 2 e$, one gets
\begin{eqnarray*}
\lefteqn{
  (2d^2-3e^2+12e(1+j)-4(3+8j+4j^2)+d(3e-2-2j)  } \\
  & \ge  &   2(2j+3)^2-3e^2+12e(1+j)-4(3+8j+4j^2)+(2j+3)(3e-2-2j) >0.
\end{eqnarray*}
We conclude that $\dim [R/(I,\ell)]_{2d-1+j} > 0$.

Now we  compute the right-hand side of (\ref{to show3}). Lemma \ref{lem:expected-Hilb} provides
\begin{eqnarray*}
\lefteqn{
\dim [R/I]_{2d-1+j} - \dim [R/I]_{2d-2+j} } \\[1ex]
&  = &
 h_A(2d-1+j) - h_A (d-1+j-e) - h_A(2d-2+j) + h_A(d-2+j-e) \\[1ex]
& = & {2d+j+2\choose 3}-5{d+j+2\choose 3}-{d+j+2-e\choose
3}+10{j+2\choose 3} .
\end{eqnarray*}
If $\dim [R/I]_{2d-1+j} - \dim [R/I]_{2d-2+j}\le 0$, then the WLP
fails because we have seen that $\dim [R/(I,\ell)]_{2d-1+j}>0$.

If $\dim [R/I]_{2d-1+j} - \dim [R/I]_{2d-2+j}> 0$,  then
\begin{eqnarray*}
\lefteqn{
  \dim [R/(I,\ell)]_{2d-1+j}-[\dim [R/I]_{2d-1+j} - \dim
[R/I]_{2d-2+j}]  } \\[1ex]
& \ge & {2d-4-5j+2e\choose 3 } - 5{d-2-3j+e\choose 3}-
 {d-2-3j+2e\choose 3}  + 5 \cdot \binom{e-j}{3} \\[1ex]
& &  - \left [{2d+j+2\choose 3}-5{d+j+2\choose
3}-{d+j+2-e\choose 3}+10{j+2\choose 3} \right ] \\[1ex]
  & = &    {3j-e+4\choose 3} > 0.
\end{eqnarray*}
Hence, we conclude in every case that $R/I$ fails the WLP in degree
$2d-1+j$, which is what we wanted to prove.

The proof of (iii$'$)  is completely analogous.

Now we prove (iv$'$). Suppose that $d$ is odd and $d-1 \leq e \leq \frac{3d-7}{2}$.  We claim that $R/I$ fails the WLP (usually by failing injectivity) from degree $\frac{5d-7}{2}$ to degree $\frac{5d-5}{2}$.  We first consider (by applying a sequence of cubo-cubic Cremona transformations (see Lemma \ref{lem:Cremona})
\[
\begin{array}{rcl}
\dim [R/(I,\ell)]_{\frac{5d-5}{2}} & = & \dim_k \mathcal L_3 \left (\frac{5d-5}{2}; \ (\frac{3d-3}{2})^5, \frac{3d-3}{2}-e \right ) \\ \\
& = & \dim_k \mathcal L_3( \frac{3d-3}{2}; \  (\frac{d-1}{2})^4,  \frac{3d-3}{2}, \frac{3d-3}{2}-e) \\ \\
& = & \dim_k \mathcal L_3( \frac{3d-3}{2}; \ \frac{3d-3}{2}, (\frac{d-1}{2})^4,   \frac{3d-3}{2}-e) \\ \\
& = & \dim_k \mathcal L_2 \left (\frac{3d-3}{2}; \ (\frac{d-1}{2})^4,  \frac{3d-3}{2}-e \right ) \ \hbox{ \ \ \ (by Proposition \ref{reduce proj sp})}.
\end{array}
\]
Notice that since $d-1 \leq e$, we have $\frac{3d-3}{2} -e \leq \frac{d-1}{2}$.   Thanks to Lemma \ref{indep cond}, these fat points impose independent conditions, so  we have
\[
\begin{array}{rcl}
\dim [R/(I,\ell)]_{\frac{5d-5}{2}} & = & \binom{\frac{3d-3}{2} +2}{2} - 4 \cdot \binom{\frac{d-1}{2} +1}{2} - \binom{\frac{3d-3}{2}-e +1}{2} \\[1ex]
& = &  \frac{9d^2-1}{8} - 4 \cdot \frac{d^2-1}{8} - \frac{1}{2} \left [ \frac{9d^2 -12d + 3}{4} - e(3d-2) + e^2 \right ] \\[1ex]
& = & \frac{e}{2}(3d-2-e) - \binom{d-1}{2}+1
\end{array}
\]
(The fact that $\dim_k \mathcal L_3( \frac{3d-3}{2}; \ \frac{3d-3}{2}, \frac{d-1}{2}, \frac{d-1}{2}, \frac{d-1}{2}, \frac{d-1}{2},  \frac{3d-3}{2}-e) = \frac{e}{2}(3d-2-e) - \binom{d-1}{2}+1$ could also be obtained via Theorem \ref{thm:fat-points-P3}).

Combined with Lemma \ref{lem:expected-Hilb}, we obtain
\[
\dim[R/(I,\ell)]_{\frac{5d-5}{2}} - \left (\dim [R/I]_{\frac{5d-5}{2}} - \dim [R/I]_{\frac{5d-5}{2}-1} \right ) =
\binom{\frac{3d-7}{2}-e+3}{3}.
\]
If the part in parentheses is non-negative then we expect injectivity, but the positivity of the part on the right (since we assumed $e \leq\frac{3d-7}{2}$) implies that the WLP fails.
Now suppose that the part in parentheses is negative, so that we expect surjectivity.  Then one checks that
\[
\begin{array}{rcl}
\dim[R/(I,\ell)]_{\frac{5d-5}{2}} & = & \frac{e}{2}(3d-2-e) - \binom{d-1}{2}+1 \\
& = & \frac{(e-d)(2d-e-2)}{2} + \binom{d+1}{2} \\
& \geq & \frac{(e-d)(d+3)}{4} + \binom{d+1}{2} \hbox{\ \ \ \ \ \ (since $e \leq \frac{3d-7}{2}$)} \\
& > & 0
\end{array}
\]
(the last inequality is for $d-1 \leq e \leq \frac{3d-7}{2}$), and so surjectivity (and hence WLP) fails.
This completes the proof of (iv$'$).

For (v$'$), we show that $R/I$ fails the WLP from degree $\frac{5d-8}{2}$ to degree $\frac{5d-6}{2}$.
Since we will use part of the computations also for (vi$'$), we consider all even $d$ such that $d-2 \le e \le \frac{3d-8}{2}$.  In the same way as above we get
\[
\begin{array}{rcl}
\dim [R/(I,\ell)]_{\frac{5d-6}{2}} & = & \dim_k \mathcal L_3 (\frac{5d-6}{2}; \ (\frac{3d-4}{2})^5,  \frac{3d-4}{2}-e) \\
& = & \dim_k \mathcal L_3(\frac{3d-2}{2}; \ (\frac{d}{2})^4,  \frac{3d-4}{2}, \frac{3d-4}{2}-e ).
\end{array}
\]
Using Theorem \ref{thm:fat-points-P3}, we obtain
\[
\dim[R/(I,\ell)]_{\frac{5d-6}{2}} = -d^2 + 3de -e^2 +6d -4e-4 = \frac{5d^2}{4} - \left ( \frac{3d-4}{2}-e \right )^2,
\]
which is at least $d^2$ in the given range for $e$.  In particular, we are finished whenever $\dim [R/I]_{\frac{5d-6}{2}} \leq \dim[R/I]_{\frac{5d-6}{2}-1}$.

If $\dim [R/I]_{\frac{5d-6}{2}} \geq \dim[R/I]_{\frac{5d-6}{2}-1}$ (so we expect injectivity),
then Lemma \ref{lem:expected-Hilb} provides, minding the bound on $e$,
\begin{equation}
  \label{eq:d even}
\dim[R/(I,\ell)]_{\frac{5d-6}{2}} - \left (\dim [R/I]_{\frac{5d-6}{2}} - \dim [R/I]_{\frac{5d-6}{2}-1} \right ) =
\binom{\frac{3d-10}{2}-e+3}{3}.
\end{equation}
Since we are assuming that the part in parentheses is non-negative, the positivity of the part on the right implies that WLP fails if $e \leq\frac{3d-10}{2}$, as claimed.

We now prove (vi$'$).  Arguing as in Theorem 3.5(iii), we see that if $t\le \lambda -2$ or $t\ge \lambda $ then
 $(\times \ell) : [R/I]_t \rightarrow [R/I]_{t+1}$ has maximal rank. So, it only remains to study the case $t=\lambda -1$. We are going to prove that the multiplication  $(\times \ell) : [R/I]_{\lambda -1} \rightarrow [R/I]_{\lambda }$ is injective, i.e. $\dim [R/(I,\ell)]_\lambda =  \dim[R/I]_\lambda - \dim[R/I]_{\lambda-1}$.

Notice that the assumptions force $d \geq 4$, so we may apply the computations of (v$'$). Thus, the desired follows by Equation \eqref{eq:d even}.
\end{proof}

The same argument gives us the following result.

\begin{proposition} \label{five var non uniform case2}
Let $L$ be a general linear form and let $I = \langle x_1^{a_1},x_2^{a_2},x_3^{a_3},x_4^{a_4},x_5^{a_5},L^{a_6} \rangle$. Assume that $5\le a_1\le a_2\le a_3 \le a_4\le a_5\le a_6\le a_1+2$. Then the ring $R/I$ fails the WLP.
\end{proposition}

Our methods extend beyond the results mentioned above.

\begin{example}
Using the above notation, one shows:
\begin{itemize}
    \item[(i)] If $d\ge 4$ and  $(a_1,a_2,a_3,a_4,a_5,a_6)=(d,d+1,d+2,d+3,d+4,d+5)$, then the ring $R/I$ fails the WLP in degree $2d+4$.
      \item[(ii)] If $d\ge 4$ and $(a_1,a_2,a_3,a_4,a_5,a_6)=(d,d+3,d+4,d+7,d+7,d+10)$, then the ring $R/I$ fails the WLP in degree $2d+9$.
      \end{itemize}
\end{example}

\bigskip


\section{Uniform powers of linear forms} \label{uniform-powers}

In this section we consider the case of an almost complete intersection of general linear forms, whose generators all have the same degree, $d$.  We give a complete answer, in the case of an even number of variables, to the question of when the WLP holds.  Interestingly, the case of an odd number of variables is more delicate, and we are only able to give a partial result, concluding with a conjecture.

We first consider the case of an even number of variables.  Set $R = k[x_1, \cdots ,x_{2n}]$, where $k$ is a field of
characteristic zero. We  determine  when an ideal
generated by uniform  powers of $2n+1$ general linear forms in $R$
fails the WLP. If $n = 1$, then $R/I$ always has the WLP due to \cite{HMNW}.
If $n = 2$, then  $R/I$ fails the WLP if and only if $d \geq 3$ by Theorem \ref{four var even}.

\begin{theorem}
  \label{2n var uniform case}
Let $L \in R$ be a general linear form, and let $I = \langle x_1^{d},
\cdots ,x_{2n}^d,L^{d} \rangle$, where $n \ge 3$.
Then the ring $R/I$ fails the WLP if and only if $d>1$.
\end{theorem}

\begin{proof} It is clear that for $d=1$, $R/I$ has WLP.

Assume $d\ge 2$. We will show that $R/I$ fails WLP in degree $nd-n$. To this end, we take  $\ell \in R$,   a general linear form. As mentioned in Section \ref{gen appr}, it is enough to show that
\begin{equation}
   \label{toshow2n}
\dim_k [R/(I,\ell)]_{nd-n} > \big [\dim_k [R/I]_{nd-n} -
\dim_k [R/I]_{nd-n-1} \big ]_+.
\end{equation}

First, we compute the left-hand side of (\ref{toshow2n}).

\vskip 2mm  \noindent {\bf Claim 1:} $\dim [R/(I,\ell)]_{nd-n}=1.$

\vskip 2mm  \noindent {\em Proof of Claim 1:} By Proposition \ref{reduce proj sp} and Theorem \ref{thm:inverse-system}, we have
\begin{eqnarray*}
\lefteqn{
  \dim_k [R/(I,\ell)]_{nd-n} } \\
& = &   \dim_k [\fq _1^{(n-1)d-(n-1)} \cap \fq _2^{(n-1)d-(n-1)} \cap \cdots
 \cap \fq_{2n}^{(n-1)d-(n-1)} \cap \fq_{2n+1}^{(n-1)d-(n-1)}]_{nd-n} \\
& = & \dim_k {\mathcal L}_{2n-2}(nd-n; ((n-1)d-(n-1))^{2n+1}).
\end{eqnarray*}
Applying Lemma \ref{lem:Cremona},  we get
\begin{eqnarray*}
\lefteqn{
\dim_k {\mathcal L}_{2n-2}(nd-n; ((n-1)d-(n-1))^{2n+1})   } \\
& = &   \dim_k
{\mathcal L_{2n-2}}((n-1)d-(n-1);
((n-1)d-(n-1))^2,((n-2)d-(n-2))^{2n-1}).
\end{eqnarray*}
Using  twice Proposition 2.3, it follows that
\begin{eqnarray}
  \label{eq:inductive}
\lefteqn{
\dim_k {\mathcal L}_{2n-2}(nd-n; ((n-1)d-(n-1))^{2n+1})   }  \nonumber \\
& = &   \dim_k {\mathcal
L_{2n-3}}((n-1)d-(n-1); (n-1)d-(n-1),((n-2)d-(n-2))^{2n-1}) \\
& = & \dim_k
{\mathcal L_{2n-4}}((n-1)d-(n-1); ((n-2)d-(n-2))^{2n-1}). \nonumber
\end{eqnarray}
If $n = 3$, then we get by applying  again Lemma \ref{lem:Cremona}
\[
\dim_k {\mathcal L_2}(2d-2; (d-1)^5) =
\dim_k {\mathcal L_2}(2d-2; (d-1)^2, 0^3) = 1,
\]
as desired.

If $n \geq 4$, then we conclude by induction using Equation \eqref{eq:inductive}.
Thus, Claim 1 is shown.\smallskip

Second, we consider the right-hand side of Inequality \eqref{toshow2n}. Taking into account Claim 1, we see that $R/I$ fails the WLP, once we have shown the following statement.

\vskip 2mm  \noindent {\bf Claim 2:} $\dim_k [R/I]_{nd-n} \le
\dim_k [R/I]_{nd-n-1}$.

\vskip 2mm  \noindent {\em Proof of Claim 2:} We use induction on $n \ge 3$.
Let $A = R/\langle x_1^{d}, \cdots ,x_{2n}^d \rangle$. Assume $n = 3$. Lemma \ref{lem:expected-Hilb}   provides
\begin{eqnarray*}
\lefteqn{
\dim_k [R/I]_{3d-3} - \dim_k [R/I]_{3d-4}   } \\[1ex]
& = &   h_A(3d-3) - h_A (2d-3) - h_A(3d-4) + h_A(2d-4) \\[1ex]
& = &  {3d+1\choose 4}-7{2d+1\choose 4}+21{d+1\choose 4} \\
& = & - \frac{1}{12} d (d - 2) (5 d^2 + 2d + 5) \\
& \le & 0,
\end{eqnarray*}
as desired.

Let $n \ge 4$. By Lemma \ref{lem:expected-Hilb}, Claim 2 can be rewritten as
\begin{equation}
   \label{eq:ineq-ci}
  h_A (n d - n) - h_A (n d - n - 1) \le h_A (n d - n -d ) - h_A (n d - n -d - 1).
\end{equation}
Consider now the ring $B = k[x_1,\ldots,x_{2n-1}]/\langle x_1^{d}, \cdots ,x_{2n-1}^d \rangle$.
Then $A \cong B \otimes_k k[x]/(x^{d})$, which implies
\begin{equation}
  \label{eq:compare-hilb}
  h_A (j) = h_B (j) + h_B (j-1) + \cdots + h_B (j  - (d-1)).
\end{equation}
Thus, the last inequality becomes
\begin{equation*}
  h_B (n d - n) - h_B (n d - n - d) \le h_B (n d - n - d) - h_B (nd - n - 2 d).
\end{equation*}
The Hilbert function of $B$ is symmetric about $\frac{1}{2} (2n-1) (d-1)$, so $h_B (n d - n) = h_B (n d - n - d + 1)$. Thus, we have to show
\begin{equation*}
  h_B (n d - n - d + 1) - h_B (n d - n - d) \le h_B (n d - n - d) - h_B (nd - n - 2 d).
\end{equation*}
To this end put  $C = k[x_1,\ldots,x_{2n-2}]/\langle x_1^{d}, \cdots
,x_{2n-2}^d \rangle$. Then $B \cong C \otimes_k k[x]/(x^d)$. Hence
using a relation similar to Equation \eqref{eq:compare-hilb}, Claim
2 follows, once we have shown
\begin{eqnarray}
  \label{eq:desired-ineq}
\lefteqn{
h_C  (n d - n - d + 1) - h_C (n d - n - 2 d + 1) } \nonumber \\
& \le &  h_C (n d - n - d) + h_C (n d - n - d) + \cdots + h_C (n d - n - 2 d + 1) \\
 && - \big [ h_C (nd - n - 2 d) + h_C (nd - n - 2d -1) + \cdots + h_C (nd - n - 3d + 1) \big ]. \nonumber
\end{eqnarray}
Our induction hypothesis (see Inequality \eqref{eq:ineq-ci}) provides
\begin{equation*}
  h_C( n d - n - d + 1) - h_C( n d - n - d) \le h_C( n d - n - 2d + 1) - h_C( n d - n - 2d ).
\end{equation*}
Since the Hilbert function of $C$ is unimodal with peak in degree $(n-1)(d-1)$, we have the following estimates
\begin{eqnarray*}
  h_C( n d - n - d) - h_C (n d - n - 2 d + 1) & \le &  h_C( n d - n - d) - h_C( n d - n - 2 d - 1) \\
0 & \le & h_C( n d - n - d - 1) - h_C( n d - n - 2 d - 2) \\
& \vdots & \\
0 & \le & h_C( n d - n - 2 d) - h_C( n d - n - 3 d - 1).
\end{eqnarray*}
Adding the last inequalities, we get the desired Inequality \eqref{eq:desired-ineq}, which completes the proof of Claim 2, and we are done.
\end{proof}

\begin{remark}  The above theorem  proves half of Conjecture 5.5.2 in
\cite{hss}. Indeed, it proves it when the number of variables is
even.
\end{remark}

\begin{remark} In \cite{hss}, Theorem 5.2.2, Harbourne,  Schenck and
Seceleanu have recently given an alternative proof of the above
theorem for $d\gg 0$..
\end{remark}

Claim 2 can be restated as a result about the growth of the
coefficients of a certain univariate polynomial.

\begin{proposition}
  \label{prop:coefficients-growth}
Let $n \ge 3$ and $d \ge 1$ be integers. Consider the univariate polynomial
\[
a_0 + a_1 z + \cdots + a_{2nd} z^{2nd} := \big (1 + z + \cdots + z^d)^{2n}.
\]
Then
\[
a_{nd} - a_{nd-1} \leq a_{nd - d - 1} - a_{nd-d-2}.
\]
\end{proposition}

\begin{proof}
Note that $a_i$ is the Hilbert function of $k[x_1,\ldots,x_{2n}]/\langle x_1^{d+1}, \cdots ,x_{2n}^{d+1} \rangle$ in degree $i$. Hence Inequality \eqref{eq:ineq-ci} establishes the claim.
\end{proof}

We now turn to an odd number of variables.  We are not able to give a result as comprehensive as that of an even number of variables, and we only consider the case of seven variables.

\begin{theorem}
  \label{thm:7 variables}
Let $L \in R = k[x_1,
\cdots ,x_{7}]$ be a general linear form, and let $I = \langle x_1^{d},
\cdots ,x_{7}^d,L^{d} \rangle$.
If $d=2$ then the ring $R/I$ has the WLP.  If $d \geq 4$ then  $R/I$ fails the WLP.
\end{theorem}

\begin{proof}
If $d=2$, we have verified on \cocoa\ that $R/I$ has the WLP.  In fact, \cocoa\ has also given the result that when $d=3$, $R/I$ fails to have the WLP because of the failure of injectivity.  However, for the failure of the WLP, a computer verification is not enough, since it is impossible to justify that the linear forms are ``general enough."  We conjecture that WLP also fails for $d=3$.

We now assume that $d\geq 4$.  We will show the failure of surjectivity in a suitable degree.  Let $\ell \in R$ be a general linear form. Let $j = \left \lfloor \frac{17}{5} (d-1) \right \rfloor$.  We want to compute
\[
\dim_k [R/(I,\ell)]_j = \dim_k \cL_5 (j; (j+1-d)^8)
\]
and in particular, to show that this dimension is non-zero.  Using Lemma \ref{lem:Cremona} four times we get
\begin{eqnarray}
  \label{eqn:lin-system 7 var}
\lefteqn{
  \dim_k [R/(I,\ell)]_{j} }  \nonumber \\
& = &   \dim_k \cL_5 (- j + 6 (d-1); (j+1-d)^2, (-j+5(d-1))^6) \\ \nonumber
& = & \dim_k \cL_5 (- 3j + 12 (d-1); (-j+5(d-1))^4, (-3 j+11(d-1))^4, ) \\ \nonumber
& = & \dim_k \cL_5 (- 5j + 18 (d-1); (-3 j+11(d-1))^6, (-5j+17(d-1))^2) \\ \nonumber
& = & \dim_k \cL_5 (- 7j + 24 (d-1); (-5j+17(d-1))^8).  \nonumber
\end{eqnarray}
This computation is correct and has a chance of resulting in a non-empty linear system if
\[
0 \le -5j+17(d-1) < - 7j + 24 (d-1),
\]
which is true since $j \leq \frac{17}{5} (d-1)$.

Thus, we distinguish the following five cases, where $e$ is an integer.

{\em Case 1}: $d-1 = 5 e$, thus $j = 17 e$.

{\em Case 2}: $d-1 = 5 e + 1$, thus $j = 17 e + 3$.

{\em Case 3}: $d-1 = 5 e + 2$, thus $j = 17 e + 6$.

{\em Case 4}: $d-1 = 5 e + 3$, thus $j = 17 e + 10$.

{\em Case 5}: $d-1 = 5 e + 4$, thus $j = 17 e + 13$.
Computation \eqref{eqn:lin-system 7 var} then shows that $\dim_k [R/(I,\ell)]_j$ equals
\begin{eqnarray*}
  \hbox{Case 1:} && \dim_k \cL_5 (e; 0^8) \\
  \hbox{Case 2:} && \dim_k \cL_5 (e+3; 2^8) \\
  \hbox{Case 3:} && \dim_k \cL_5 (e+6; 4^8) \\
  \hbox{Case 4:} && \dim_k \cL_5 (e+2; 1^8) \\
  \hbox{Case 5:} && \dim_k \cL_5 (e+5; 3^8).
\end{eqnarray*}
It is clear that these linear systems are not empty if $e\geq 0$, thus $d\geq1$.

To prove failure of the WLP in degree $j$ it remains to check that the expected dimension is zero. Using Lemma \ref{lem:expected-Hilb}, we obtain
\begin{eqnarray*}
\lefteqn{
\dim_k [R/I]_{j} - \dim_k [R/I]_{j-1}   } \\[1ex]
& = &  \binom{j+5}{5} - 8 \binom{j+5-d}{5} + 28 \binom{j+5-2d}{5}  - 56 \binom{j+5-3d}{5}.
\end{eqnarray*}
Notice that the last binomial coefficient is zero if $d \leq 10$, while the third one is zero for $d \leq 2$.  However, the computations of the polynomials below are not affected.
Distinguishing the five cases above, this dimension times $5!$ equals
\begin{eqnarray*}
  \hbox{Case 1:} && -101 995\  e^5  -69925\ e^4-15975\ e^3+565\ e^2+730\ e+120 \\
  \hbox{Case 2:} &&  -101 995\ e^5  -139850\ e^4-60225\ e^3-1330\ e^2+5080\ e+960\\
  \hbox{Case 3:} && -101 995\ e^5 -209775\ e^4-133975\ e^3-8145\ e^2+19730\ e+5040 \\
  \hbox{Case 4:} && -101 995\ e^5 -359875\ e^4-499175\ e^3-336365\ e^2-107910\ e-12600 \\
  \hbox{Case 5:} && -5\ (e+1)\ (20399\ e^4+65561\ e^3+74044\ e^2+32716\ e+3840)
\end{eqnarray*}
It is clear that the first three  polynomials are negative if $e \ge 1$, while the last two are negative whenever $e \ge 0$.  Thus the expected dimension is zero whenever $d \geq 4$.  In particular, we have shown the failure of surjectivity (in particular the failure of the WLP) for $d \geq 4$.
\end{proof}

In trying to apply the approach of Theorem \ref{thm:7 variables} to the general case, we were able to mimic the choice of $j$, as well as the main details of the proof that $\dim [R/(I,\ell)]_j > 0$.  However, a proof of the required inequality to verify that it is surjectivity rather than injectivity that fails eluded us.  Based on experiments with \cocoa, we end with the following conjecture (also to complete the case of seven variables).  Notice that the case $d=2$ has the WLP in seven variables (as noted above).

\begin{conjecture} \label{conj uniform}
Let $R = k[x_1,\dots,x_{2n+1}]$, where $n \geq 4$.
Let $L \in R$ be a general linear form, and let $I = \langle x_1^{d},
\cdots ,x_{2n+1}^d,L^{d} \rangle$.
Then the ring $R/I$ fails the WLP if and only if $d>1$.  Furthermore, if $n=3$ then $R/I$ fails the WLP when $d=3$.
\end{conjecture}

\bigskip

\noindent{\bf Acknowledgements:} Shortly before the end of the
writing of this paper we were informed that  B.\ Harbourne, H.\
Schenck and A.\ Seceleanu were able to prove Theorem \ref{2n var
uniform case} for $d\gg 0$. We are very grateful to H.\ Schenk for
sending their paper to us, and for the helpful discussions we have
had. We also thank M.\ Boij for interesting and useful
discussions.

\end{document}